\documentclass[a4paper,12pt]{article}

\usepackage{etex}

\usepackage[utf8]{inputenc}

\usepackage[T1]{fontenc}
\usepackage[english]{babel}
\usepackage{amsmath, latexsym, amssymb, amscd,
pictex,amsthm,amstext,amsxtra,amsopn,amsbsy, enumerate}
\usepackage{subfig}
\usepackage{graphicx}
\usepackage{natbib}

\usepackage{algorithm} 
\usepackage{algpseudocode} 

\makeatletter
\newcommand{\pushright}[1]{\ifmeasuring@#1\else\omit\hfill$\displaystyle#1$\fi\ignorespaces}
\newcommand{\pushleft}[1]{\ifmeasuring@#1\else\omit$\displaystyle#1$\hfill\fi\ignorespaces}
\makeatother

\newtheorem{thm}{\bf Theorem}[section]
\newtheorem{prop}[thm]{\bf Proposition}
\newtheorem{property}{\bf Property}

\newtheorem{lem}[thm]{\bf Lemma}

\newtheorem{ex}{\bf Example}[section]

\newcommand{\Crossroad}{\mathcal{C}}
\newcommand{\oCrossroad}{\mathcal{C}^o}
\newcommand{\aCrossroad}{\mathcal{C}^r}
\newcommand{\crossroad}{\mathfrak{c}}
\newcommand{\boundary}{\mathcal{B}}
\newcommand{\Cuts}{\mathrm{Cuts}}

\DeclareMathOperator{\Vol}{Vol}
\DeclareMathOperator*{\arginf}{arginf}

\title{How many T-tessellations on $k$ lines? Existence of associated Gibbs measures on bounded convex domains}
\author{Jonas KAHN}
\date{\today}

\begin{document}
\maketitle

\begin{abstract}

The paper bounds the number of tessellations with T-shaped vertices
on a fixed set of $k$ lines: tessellations
are efficiently encoded, and algorithms retrieve them, proving injectivity. 
This yields existence of a completely random T-tessellation, as defined by \citet{KK}, and of its Gibbsian modifications.
The combinatorial bound is sharp, but likely pessimistic in typical cases.






\end{abstract}

\section{Introduction}
\label{intro}

Some man-made landscapes, such as plots of land, may be viewed as T-tessellations, that is a tessellations of a subset of the plane, where all vertices are degree three and with one flat angle. \citet{KK} have developed a random model aiming at representing such landscapes and real-world structures with similar geometry. This article proves that their \emph{completely random T-tessellation} (CRTT) and its Gibbsian extensions are well-defined.

However, the meat of the article will be a study of the number of tessellations on $k$ given lines. An upper bound is the necessary ingredient to prove existence of CRTT.

\bigskip

T-tessellations first appear as a special case of polygonal Markov fields by \citet{Arak&al}. They have introduced a very general model for random planar graphs, directly defined by their measure on the set of graphs.
The measure depends on an energy function that can be specified to yield T-tessellations, as detailed by \citet{Mackisack&Miles2}. These graphs have very nice mathematical properties. As a result, they may be sampled exactly, without resorting to Metropolis algorithm. 

\citet{Thale} offers a variation on those tessellations, by allowing the Poisson line process on which the segments are built to be inhomogeneous and anisotropic. The resulting tessellation may be sampled exactly in the same direct way. Thäle focuses on existence of the whole plane tessellation, and the statistical properties of the edges and cells of the tessellation.

However these models' nice properties come at a price: the tessellations are necessarily very random, with many elements that behave as if they were independent. For example, the intersection with any line is a Poisson process.

Another category of models
with T-vertices had already been studied \citep{Mackisack&Miles}. They study rectangular tessellations, built  from random point seeds from which segments grow until they are blocked, as in the model by \citet{Gilbert} for mudcracks.  

\cite{Cowan} has introduced a family of models based on successive divisions of cells, including in particular the STIT tessellations by \citet{Nagel&Weiss}. It allows quite some flexibility but can only generate tessellations that can be recursively built by dividing cells, and this a subset of all tessellations.

\citet{KK} have introduced a model of completely random T-tessellations (CRTT). They then take Gibbsian modifications to make it very flexible. The CRTT is characterized by its very simple 
Papangelou kernel, making it similar to a Poisson process. Heuristically, the ratio of probability density between a tessellation $T$ and the same tessellation with an added segment $s$ does not depend on the tessellation, as long as we keep a T-tessellation when $s$ is added.
In particular, all tessellations with the same number of segments have the same density, unlike in the models by \citet{Arak&al} and \citet{Thale}. The model admits (stable) Gibbsian modifications, making it very flexible. Different energies allow different kinds of landscapes. We may for example require that all parcels have a similar area, or penalise sharp angles. The price to pay is harder sampling, requiring Monte-Carlo Markov chain algorithms.

\medskip

From a theoretical point of view, \citet{KK} have not managed to prove that their measure was 
finite
,which is necessary for the model to be well-defined. This article focuses on proving that we have a true probability measure. Therefore, we shall only consider T-tessellations on bounded domains. Existence of the whole plane tessellation is outside the scope of the article.

The results are achieved through purely combinatorial means: we shall define an encoding on tessellations, bound the number of different outputs, and prove we can rebuild the tessellations from the encoding. Existence of the CRTT and its Gibbsian modifications are both quite easy from that point.

\medskip

Section~\ref{submain} contains the minimal notations to state the main result, Theorem~\ref{combi}, and the strategy of proof with a few comments: we bound the number of different T-tessellations on a given configuration of lines by devising algorithms that rebuild the tessellation from their input, and counting the number of different inputs they can have. Section~\ref{secCRTT} gives the motivation: description of the CRTT and its modifications, and  proof that Theorem~\ref{combi} yields their existence. This section may be skipped without hurting understanding of the other parts. Section~\ref{toutesnot} contains the thorough notations for T-tessellations necessary to write, describe and analyse the algorithms.  
The algorithms are detailed in Section~\ref{algos}, together with the less cumbersome proofs. Formal proofs of correctness of the algorithms are delayed to Appendix \ref{proofs}. The algorithms themselves, and figures illustrating them, are found in Appendix~\ref{alg_and_fig}.  Finally we shall discuss the limitations of the proof and hint at possible improvements in Section~\ref{perspectives}.

\section{Main result, strategy and motivation}
\label{res}

\subsection{Main theorem}
\label{submain}

Let $W$ be a compact convex domain in the plane, with non-empty interior. A finite polygonal tessellation of $W$ is a finite partition of $W$ into convex sets with disjoint interiors, called \emph{cells}, such that the boundary of each cell is a union of a finite number of line segments, called (inner) \emph{edges} and parts of the boundary $\boundary$ of $W$ itself. A \emph{segment} is a maximal union of aligned and contiguous inner edges. A \emph{vertex} of a tessellation is T-shaped if it is of degree three and two of the incident edges are aligned, or if it is of degree one on the boundary of $W$. A \emph{T-tessellation} is a polygonal tessellation such that:
\begin{itemize}
\item{All vertices in $W$ are T-shaped.}
\item{No two segments are aligned.}
\end{itemize}

We denote by $\mathcal{T}_W $ the set of T-tessellations on $W$. We denote by $L(T)$, or simply $L$ the set of lines that support the segments of $T$. Conversely, we denote by $\mathcal{T} (L) = \left\{ T: L(T) = L \right\}$  the set of T-tessellations whose segments are supported by $L$, with one segment on each line in $L$.

\bigskip

For a set of $k$ lines $L$, there are \emph{a priori} several different T-tessellations $T$ whose segments are supported by $L$. How many exactly depends on the precise set of lines $L$. The main result of this paper is an upper bound on $\mathcal{N} (k)$ the maximum number of different T-tessellations whose segments are supported on $k$ given lines:
\begin{align*}
    \mathcal{N} (k) = \sup_{L: \#L=k} \#\mathcal{T} (L). 
\end{align*}

\begin{thm}
\label{combi}
For any set of $k$ lines, for any $\epsilon > 0$ the number of T-tessellations built on them is at most:
\begin{align}
\label{count}
\mathcal{N} (k) \leq C^k \left(\frac{k}{(\ln k)^{1-\epsilon }}\right)^{k - k / (\ln k)},
\end{align}
where $C$ depends only on $\epsilon $.
\end{thm}

In particular, for all $a > 0 $: 
\begin{align}
\label{Nksup}
\mathcal{N} (k) = o(k^k a^k).
\end{align}

To bound this number $\mathcal{N} (k)$, we want to find a description of a tessellation, that may be used as input to an algorithm. We shall show that the algorithm then rebuilds the initial tessellation $T$. In other words, there is an injective function from the descriptions to the tessellations.  A bound on the number of different descriptions then yields a bound on $\mathcal{N} (k)$. Many of the notations in Section \ref{toutesnot} are also devised to be easy of use within an algorithm.

In particular, multiplying the number of possible descriptions by an exponential makes no difference in proving bound \eqref{Nksup}, given its form. So that we may 
 add to the description any element that takes at most an exponential number of values $b^k$. We say such an element is \emph{free}.

Two typical examples of free elements we shall use are:
\begin{ex}
    \label{ex1}
    a subset of the $k$ lines: there are $2^k$ different ones.
\end{ex}
\begin{ex}
    \label{ex2}
a function\footnote{We use the convention $\mathbb{N} = \left\{ 0, 1, \dots \right\}$.} $f: L \cup \boundary \to \mathbb{N} $ with $\sum_{l\in  L \cup \boundary} f(l) \leq k$. This corresponds to splitting at most $k$ indiscernible objects among the $k$ lines and the boundary. The number of possibilities is $ \sum_{l\leq k} {k + l \choose k} = {2k+1 \choose k} \leq 4^k$.
\end{ex}

\subsection{Motivation: Completely random T-tessellations}
\label{secCRTT}

Let us briefly motivate bound~\eqref{Nksup} and introduce the CRTT model by \citet{KK}.

We first need some definitions related to Poisson line processes.  More details on Poisson line processes and generalisations  may be found in the book by \citet{Schneider&Weil}. 
We may define a line in $\mathbb{R} ^2$ by an angle $\alpha \in [0, \pi[$ and a distance $p \in \mathbb{R} $. Let us consider the origin $O = (0,0)$ and the point $P$ of radial coordinates $(p, \alpha )$. Then the line $D(p, \alpha )$ is the line orthogonal at point $P$ to the line $(OP)$.

        We now consider $D^{-1}\left( \mathcal{L} _{\mathring{W}} \right) $ the preimage by $D$ of the set $\mathcal{L} _{\mathring{W}}$ of all lines in $\mathbb{R} ^2$ that intersect the interior of $W$.
A Poisson line process $\lambda_W$ on $W$ is the image by $D$ of a Poisson point process on $ D^{-1}\left( \mathcal{L} _{\mathring{W}} \right)$. If the point process is simple, then the line process is a random measure of the form $ \sum_{i=1}^N \delta_{L_i}$
where $N\in \mathbb{N} $ is a random variable and all the lines $L_i \in \mathcal{L} _{\mathring{W}} $ are almost surely distinct. So that the line process may be viewed as the law of a random finite set of lines $\left\{ L_i \right\} _{1\leq i \leq N} $ that intersect $\mathring{W}$.

In particular, let us write  $L_W^{\tau}$ for the set of lines with law $\lambda _W^{\tau} $, the image of the Poisson point process on $D^{-1}\left( \mathcal{L} _{\mathring{W}} \right) $  with intensity $\frac{\tau}{\Vol\left(D^{-1}\left( \mathcal{L} _{\mathring{W}} \right)\right)}\Vol$, where $\Vol$ is the Lebesgue measure. Notably, the cardinal of  $L_W^{\tau}$  is a Poisson variable with parameter $\tau$:
\begin{align}
    \label{PoissonCard}
    \mathbb{P}\left[\#L^{\tau}_W = k\right] & = \exp(-\tau) \frac{\tau ^{k}}{k!}.
\end{align}

\bigskip

\citet{KK} define the CRTT by:
\begin{align}
\label{CRTT}
\mu_{CRTT}(A) & = Z^{-1} \mathbb{E}\left[\sum_{T \in \mathcal{T}(L_W^{\tau})} \mathbf{1}_A(T) \right] & \text{for $A\in \sigma (\mathcal{T}_W )$,}
\end{align}
where $Z$ is a normalising constant, the expectation is for $L_W^{\tau}$ with respect to the Poisson line process $\lambda _W^{\tau}$, and $\sigma (\mathcal{T}_W )$ is the standard hitting $\sigma $-algebra on the set of T-tessellations on W \citep[see][]{Matheron}. Intuitively, this means that each T-tessellation has a weight proportional to the weight of the set of its supporting lines $L(T)$ in the Poisson line process.

The bound~\eqref{Nksup} on the number of T-tessellations on $k$ given lines will allow to prove that the normalising constant $Z$ is indeed finite,  and thus that this CRTT was well-defined. The same calculation allows to prove existence of Gibbsian modifications for stable energy functionals, that is for probability measures defined by:
\begin{align}
\label{Gibbs}
\mu_{H}(A) & = Z_H^{-1} \mathbb{E}\left[\sum_{T \in \mathcal{T} (L)} \mathbf{1}_A(T) \exp(-H(T)) \right] & \text{for $A\in \sigma (\mathcal{T} _W)$,}
\end{align}
where $Z_H$ is a normalising constant and 
the energy is bounded from below by a linear function of the number of lines in the tessellation, that is  $  H(T)\geq C\#L(T)$ for some real constant $C$.

As a remark, we could think of using existence in the model of \citet{Arak&al} to get finiteness of $Z$. Indeed, they specialise their model of random planar graphs to T-tessellations. The idea would be to compare $Z$ for a fixed intensity $\tau$, say $1$, to their model  with very high intensity $\tau$, hoping that the density is an upper bound everywhere. However, there is a problem.  The model by \citet{Arak&al}  comes from setting
\begin{align*}
H(T) = |T| \ln 2 + \tau l(T) / b(W) ,
\end{align*}
and normalising afterwards, where $|T|$ is the number of vertices of the tessellation $T$ in the interior of $W$, $l(T)$ is the sum of the lengths of the edges of $T$, and $b(W)$ is the perimeter of $W$. As can be seen, the energy depends on $\tau$, so that the density for tessellations with many (long) edges drops when the intensity $\tau$ increases.

Let us now show how the combinatorial bound~\eqref{Nksup} implies existence of the Gibbs measure \eqref{Gibbs}:

\begin{thm}
\label{exGibbs}
Let $H(T)$ be an energy on $T$ such that $H(T) \geq  C \#L(T)$, for some real $C$ and any tessellation $T$. Then for any expected number of lines  $\tau $ in the reference Poisson line process $\lambda _W^{\tau}$, the Gibbs measure $ \mu_{H}(A) $ is well-defined and finite.
\end{thm}
 
\begin{proof}
    We have to prove that the measure is finite. We denote by $c_i$ any constant. Using the bound on $H(T)$, Stirling formula, equation~\eqref{PoissonCard} and Theorem \ref{combi}, we get:
\begin{align}
    Z_H & =  \mathbb{E}\left[\sum_{T \in \mathcal{T} (L)} \exp(-H(T)) \right] \notag \\
        & =  \int_{\mathcal{L}_W } \sum_{T \in \mathcal{T} (L)}  \exp(-H(T))   \mathrm{d} \lambda^{\tau}_W(L) \notag \\
        & \leq   \int_{\mathcal{L}_W } \exp(-C\#L(T)) \,\,\, \# \mathcal{T} (L) \,\,\,  \mathrm{d} \lambda^{\tau}_W(L) \notag \\       
        &\leq \sum_{k=0}^{\infty}  \exp( - C k) \exp \left( - \tau  \right) \frac{\tau ^k}{k!} \mathcal{N} (k) \label{Z}  \\
&\leq \sum_{k=0}^{\infty} c_1 \frac{c_2^k}{k^k}    \mathcal{N} (k) \notag\\
    &\leq \sum_{k=0}^{\infty}  2^{-k}c_3 \notag \\  
& < \infty.                                 \notag \qedhere
\end{align}
\end{proof}

As a remark, the results would translate effortlessly to any simple anisotropic and inhomogeneous underlying Poisson model, as used by \citet{Thale}. Since we only use  the expected number of lines $\tau$ in the Poisson line process, we do not care about whether the process is homogeneous isotropic or not.

\section{Notations and generalities on T-tessellations}
\label{toutesnot}

Let us have a closer look at T-tessellations.

A T-tessellation $T$ is built on a set of lines $L$ that support its \emph{segments}. Since no two segments are aligned in a T-tessellation, there is a unique segment supported by each line $l \in L$. We shall write $s(l)$ for this segment. 

The endpoints of those segments can only be an intersection with another line, or with the boundary of $W$. So that, knowing $L$, the only places where something can happen are those intersections. We call them \emph{crossroads}. A generic crossroad is denoted by $\crossroad$. When specifying the crossroad, through the lines that intersect, we write $\crossroad(l,m)$, for $l,m \in L$. Conversely, for a given crossroad $\crossroad(l_1,l_2)$, we denote the set of corresponding lines by $l(\crossroad(l_1, l_2)) = (l_1,l_2)$. Naturally the crossroad $\crossroad(l,m)$ is the same as $\crossroad(m,l)$. For parallel lines, we may define the crossroad $\crossroad(l,m)$ as a point infinity, but, apart from technicalities, we are only interested in the crossroads in $W$. A special case is when a line intersects the boundary $\boundary$. Since $W$ is convex, this happens exactly twice, so that $\crossroad(\boundary, l)$ and $\crossroad(l, \boundary)$ are different. Conventionally $\crossroad(\boundary, l) < \crossroad(l, \boundary)$ for the order we define now. We write $\Crossroad $ for the set of crossroads.

Given $k$ lines, we can find their $k(k-1) / 2$ intersections (at most), as well as their $2k$ intersections with the boundary. We choose an axis along which each crossroad has a different coordinate, except for pairs of crossroads that are the same point in the plane. Moreover, we may choose the axis to not be colinear or perpendicular to any line $l\in L$. This axis will be called the \emph{time axis}, or indifferently \emph{abscissas} axis. The corresponding coordinates are called either \emph{times} or \emph{abscissas}. We shall use the usual vocabulary associated to time, such as saying that a point (or a crossroad defined at that point) happens before another if its abscissa is smaller. We also use \emph{left} and \emph{right} for smaller and larger times.

Since the time axis is not perpendicular to any line $L$, each segment's endpoints happen at distinct times. We say that the segment is \emph{born} at its endpoint with lower time, and \emph{dies} at the other. 

Let us consider the  endpoint where the segment $s(l)$ is born. Since all vertices are T-shaped, either it is on the boundary $\boundary$, or it is in the relative interior of another segment $s(m)$. In both cases, the  point belongs to no other segment. We say that $m$ (or $\boundary$ in the former case) is the   \emph{parent} of the segment's line $l$, and the segment's line $l$ is its \emph{child}. Similarly, the  endpoint where the segment $s(l)$ dies belongs either to $\boundary$ or to the relative interior of a single other segment $s(m)$. We say that $m$ (or $\boundary$) is the killer of $l$, and $l$ its victim.


These relations thus give us two trees, the \emph{tree of births} and the \emph{tree of deaths}. Both have $k+1$ nodes, labelled as the boundary and the $k$ lines, and both are rooted at the boundary. For simplicity, we shall always speak of the nodes through their labels, saying ``the parent of a line'' instead of ``the parent of the node labelled by a line'', and so on. The parent of a line in the tree of births is its parent as defined in the former paragraph. Conversely, children in the tree of births are exactly children as defined above. The parent of a line in the tree of deaths is its killer as defined above. Children in the tree of deaths correspond to victims.

We shall from now on assume that all crossroads are distinct. Indeed, this does not change the bound on $\mathcal{N} (k)$:
\begin{lem}
    \label{distinct}
    
    If the crossroads in a set of $k$ lines $L$ are not distinct, then there is a set of $k$ lines $L'$ with distinct crossroads, such that $ \# \mathcal{T} (L) \leq \# \mathcal{T} (L')$. 
\end{lem}

Proof in appendix.

\medskip

With this, we may now order all the crossroads according to time. The ordered list of crossroads will be denoted $\oCrossroad$. The reverse-ordered list will be denoted $\aCrossroad$. And we shall often refer to a crossroad simply as its unique abcissa from now on.

We may denote by $x_b^T(l)$ and $x_d^T(l)$ the times of birth and death of the segment $s(l)$.
The tree of births encodes all information about births, that is on leftmost endpoints of the segments $\{x_b^T(l)\}_{l\in L}$. Symmetrically, the tree of kills encodes all information about deaths, that is on rightmost endpoints of the segments $\{x_d^T(l)\}_{l\in L}$. So that rebuilding the two trees is equivalent to rebuilding the tessellation.

\bigskip

Since the segments describe the tessellation, and for algorithmic purposes,  we now think of a T-tessellation on a set of lines $L$ as a 
couple of functions $T = (x_b^T, x_d^T)$ with $x_{\bullet}^T: L\cup\{\boundary\} \to \mathbb{R} $.

We have added the boundary $\boundary$ into the domain of $x_{\bullet}^T $ to make writing the algorithms easier.
For the same reason, conventionally, we now require that:
\begin{itemize}
\item{$W$ is contained in the band of abscissas $(0,1)$.}
\item{The boundary is ``always alive'': $x_b^T(\boundary) = 0$ and $x_d^T(\boundary) = 1$.}
\end{itemize}

Even when they follow these requirements, not all such couples of functions are a tessellation, let alone a T-tessellation. We shall dub \emph{prototessellation} any such couple. The notion will be mainly useful for initialisation of the algorithms.

A T-tessellation is a prototessellation $P$ with the following three properties:
\begin{itemize}
\item{Segments do not cross:
\begin{align}
\label{nocross}
\forall\, l,m \in L: \lnot &   \left[ x_b^P(l) < \crossroad(l,m) < x_d^P(l) \mbox{ and } x_b^P(m) < \crossroad(l,m) < x_d^P(m) \right]   
\end{align}
}
\item{Segments are born on the relative interior of another segment, or on the boundary:
\begin{align}
\label{born_on_segment}
\mbox{If } x_b^P(l) = \crossroad(l,m), \mbox{ then } x_b^P(m) < \crossroad(l,m) < x_d^P(m).
\end{align}
}
\item{Segments die on the relative interior of another segment, or on the boundary:
\begin{align}
\label{die_on_segment}
\mbox{If } x_d^P(l) = \crossroad(l,m), \mbox{ then } x_b^P(m) < \crossroad(l,m) < x_d^P(m).
\end{align}
}
\end{itemize}

A prototessellation where segments do not cross \eqref{nocross} is a \emph{pretessellation}. We deal with such objects within the algorithm, in some cases as output. We shall usually write $P$ for either a prototessellation or a pretessellation, and $x_b^P$ and $x_d^P$ for the times of birth and death in $P$.

\section{Algorithms and Result}
\label{algos}
\subsection{Preliminary algorithm}
\label{prem}

If we know the tree of births alone, we can almost rebuild the tessellation $T$. We only need the number of murders of each line, which is  free, since it is bounded from above by $4^k$. Hence counting the number of tessellations in the worst case $\mathcal{N} (k)$ is essentially equivalent to counting the highest possible number of trees of births on $k$ lines.

We use this fact to devise a first encoding of tessellations. It only gives finiteness of $Z_H$ for low intensities ($\tau < (4e)^{-1}$ with non-negative energy $H$). However it is a basis of our final encoding, and the proof of its efficiency introduces ideas that we shall use again, while staying in an easier context.

The input of our first algorithm is $(x_b^T, M^T)$, meaning:
\begin{itemize}
\item{We know the whole tree of births $x_b^T$.
    }
\item{We know how many lines are killed by each line: 
  \begin{align}
\label{MT}
M^T&: L  \cup \boundary  \to \mathbb{N} & M^T(l) &= \#\left\{ m: x_d^T(m) = \crossroad(m,l) \right\} .
\end{align}
  }
\end{itemize}


We may now rebuild the tessellation $T$ with Algorithm \ref{algo1}. The process is illustrated by Figure~\ref{fig_algo1}. Both are given in Appendix~\ref{alg_and_fig}.

Informally, we move along the abscissas axis, while prolongating the segments that are alive. We know when each segment is born. So, we add them to the living segments at their time of birth. When two segments cross, we look at their remaining number of murders. One of the two must be zero. The corresponding segment is killed. The other segment's number of murders is decreased by one. When a segment hits the boundary, it is also killed. When we attain the rightmost point of $W$, the tessellation is complete.

\begin{lem}
\label{lemalgo1}
Algorithm \ref{algo1} yields $T$.
\end{lem}

Proof in appendix.

Since $M^T$ is a function from $L \cup \boundary$ to $\mathbb{N} $ with $\sum_{l\in L \cup \boundary} M^T(l) \leq k$, by Example \ref{ex2}, there are at most $4^k$ possibilities for $M^T$. Since the parent of any line is another line or the boundary, there are at most $k^k$ different possible trees of birth. So that the former lemma yields $\mathcal{N} (k) \leq (4k)^k$. Putting that back into bound \eqref{Z} would yield a convergent series if $\tau < 1/(4e)$ and $C\geq 0$.

\subsection{Main algorithm}
\label{sec_main}

In the following, adjectives like ``true'' or ``real'' will always mean ``in the tessellation $T$ to be retrieved''. 

The previous encoding still uses too much information for proving existence of the CRTT with high intensity. Specifically, describing the whole tree of births dooms the effort.

Next algorithm rebuilds $T$ while knowing only part of the times of birth. The price to pay is higher complexity: instead of one pass on crossroads, we have to loop back and forth in time, prolongating orphan segments to their birth, and cutting too old segments, until stabilisation.

Let us be precise. The algorithm will take as input $(U_0, x_b^P, O^T, V)$ satisfying a list of \emph{requirements}. Specifically:
\begin{enumerate}
    \item{There is a set of \emph{orphan} lines whose parents we will not give as input. This is $U_0 \subset L$. If we do not know a line's parent, then it must have at least one child, and we must know its first child:
\begin{align}
\label{binary}
(l \in U_0) \implies \left( C(l) \hat{=} \left\{ m \in L: x_b^T(m) = \crossroad(l,m) \right\} \neq \emptyset \mbox{ and } \arginf_{m\in C(l)} x_b^T(m) \not\in U_0 \right) .
\end{align}
}
\item{The parents of the non-orphan lines are known. That is $x_b^P(l) = x_b^T(l) $ for all lines in $L\setminus U_0$. The function $x_b^P$ is otherwise undefined at input.}
\item{We know the number of \emph{orphan children} each line has in the true tessellation:
    \begin{align}
\label{Others}
O^T &: L  \cup \boundary  \to \mathbb{N}, & O^T(l) &= \#  \left\{    m \in U_0: x_b^T(m) = \crossroad(l,m)\right\}.
\end{align}
}
\item{We know the number of murders of each line in a specific pretessellation associated to the tessellation $T$. These \emph{virtual murders} are given by $V:L  \cup \boundary \to \mathbb{N} $, made precise below.}
\end{enumerate}

A quick look at the input shows that $U_0$ is a subset of the lines, and $O^T$ and $V$ are functions from $L  \cup \boundary $ to $\mathbb{N} $ with sum of all images at most $k$, so that by Examples \ref{ex1} and \ref{ex2}, they are free. On the other hand, we shall need to find a few more constraints on $x_b^P$ to prove Theorem \ref{combi}.

\bigskip

To define virtual murders $V$, we first describe a pretessellation $P(U_0)$. We shall show later on that this is the pretessellation yielded by the algorithm  after its initialisation phase:
\begin{align}
x_b^{P(U_0)}(l) & = x_b^T(l) & \mbox{if $l \not \in U_0$.} \label{birth_nU0} \\
x_b^{P(U_0)}(l) & = \inf\left\{\crossroad(l,m): \crossroad(l,m) = x_b^T(m) \right\} & \mbox{if $l \in U_0$.} \label{birthU0} \\
\label{deathU0}
x_d^{P(U_0)}(l) & = \inf \left\{ \crossroad(l,m): x_d^{P(U_0)}(m) \geq \crossroad(l,m) \geq  x_d^T(l), x_b^{P(U_0)}(m) \right\} & \mbox{for all $l \in L$.}
\end{align}
Notice that the times of death are well-defined: we list crossroads in timewise order. At each crossroad, we know if the involved lines are already dead in $P(U_0)$, and thus if the condition in the infimum is met. Furthermore, the time of death ensures that there is no crossing \eqref{nocross}, so that $P(U_0)$ is a pretessellation.

The number of virtual murders is simply the number of kills in $P(U_0)$, given that simultaneous deaths do not count:
\begin{align}
\label{Vl}
V(l) & = \# \left\{ m \in L : \crossroad(l,m) = x_d^{P(U_0)}(m) < x_d^{P(U_0)}(l) \right\} .
\end{align}

Intuitively, the \emph{virtual murders} are chosen so that, during initialisation of Algorithm \ref{main}, the lines are killed as soon as possible after their true deaths. 

\bigskip

Algorithm~\ref{main} and its subroutines Algorithms~\ref{parent} and~\ref{cuts} are in Appendix~\ref{alg_and_fig}, together with three-page long Figure~\ref{rec2}, which illustrates the process.

In the algorithm,
we want to find the parents of the orphan lines. The variable $U$ will contain the lines whose
parents we are sure we do not know yet.

Informally, we first initialise Algorithm~\ref{main} by moving along the time axis, while prolongating the segments that are deemed alive. Namely either the segment is an orphan, and we prolongate it when it has a known child, or we know its time of birth, and we prolongate it from that time. When two segments cross, we stop each one if its number of virtual murders is zero. Both may be stopped at the same time, and at least one must be. If a segment is not stopped, its number virtual of murders is decreased by one. When a segment hits the boundary, it is killed. End of initialisation.

Now we loop. Each iteration consists of an extension pass backwards in time, and a cutting loop, forwards in time. 

The pass backwards in time, or \emph{parent-seeking loop} is given in Algorithm~\ref{parent}. During the pass backwards in time, we extend the segments whose parent we do not know. Since these segments are extended before their first child \eqref{binary} in the real tessellation $T$, they never cross each other in the process. We stop extending backwards a segment when it hits another segment.


The \emph{cutting loop} is given in Algorithm~\ref{cuts}. During the cutting loop, we cut the segments who have too many children. Namely, we count the number of orphan (lines in $U_0$) children a line has, and cut when we reach its number of \emph{orphan children}.
The consequence of this operation is that its other orphan children will be further extended in the next loop iteration. 

End of loop. Stop when all numbers of orphan children are zero. End of algorithm.

\bigskip


Let us highlight a few key points about Algorithm \ref{main}:
\begin{prop}
    \label{facts}
    After initialisation, and throughout the algorithm, the prototessellation $P$ satisfies the following properties:
\begin{itemize}
\item{It is a pretessellation: segments do not cross \eqref{nocross}.}
\item{Deaths are late: 
        \begin{align}
        \label{late_death}
        x_d^P(l) & \geq x_d^T(l) & \mbox{for all $l \in L$}.
    \end{align}
}    
\item{Births are late: 
   \begin{align}
        \label{late_birth}
        x_b^P(l) & \geq x_b^T(l) & \mbox{for all $l \in L$}.
    \end{align}
  }
\item{If  we are sure we do not know yet the parent of a line, its time of birth is strictly overestimated: 
       \begin{align}
        \label{late_unknown}
        x_b^P(l) & > x_b^T(l) & \mbox{for all $l \in U$}.
    \end{align}
}
\item{Times of birth are lower than the true time of death: 
    \begin{align}
    \label{superposition}
    x_b^P(l) & < x_d^T(l) & \mbox{for all $l \in L$.}
\end{align}
}
\item{If a line $l$ has a child $m$ before its true death, then it is really its child:
    \begin{align}
    \label{realchild}
    \left(x_b^P(l) < x_b^P(m) = \crossroad(l,m) < x_d^T(l)\right) \implies \left( x_b^T(m) = \crossroad(l,m) \right). 
\end{align}
}    
    \end{itemize}
Moreover:
\begin{itemize}
    \item{Birth and death time are decreasing after preinitialisation (stage~\ref{end_preinitialisation}): 
            \begin{align}
        \notag                     & \text{\emph{If we hit a stage of the form }} x_{\bullet}(l) \gets \crossroad, \\ 
        \label{decrease}           & \text{\emph{then }}x_{\bullet}(l) > \crossroad  \text{\emph{ just before}}. 
        \end{align}
     }
    \item{At the end of the parent-seeking loop (Algorithm~\ref{parent}), all segments are born on the relative interior of another segment, or on the boundary~\eqref{born_on_segment}.}        
\end{itemize}

\end{prop}

Proof in appendix.

\bigskip

These properties ensure that the algorithm ends, and that we end up with a pretessellation similar to the true tessellation $T$:

\begin{lem}
\label{ends}
With input satisfying the requirements given at the beginning of the section, Algorithm~\ref{main} ends.

Its output is a pretessellation $P_o$ with  late births~\eqref{late_birth} and late deaths~\eqref{late_death}, with births occurring on the relint of segments or on the boundary~\eqref{born_on_segment}, and before the true death times~\eqref{superposition}. Moreover, each line has the same number of children as in the real tessellation~$T$. The children it has before its true death are real children \eqref{realchild}. In particular, if the line has the same time of death as in $T$, then its children are the same as in $T$.
\end{lem}

\begin{proof}
    In Algorithm \ref{cuts}, stages \ref{current_birth} to \ref{finsym1} and stages \ref{current_birth2} to \ref{finsym2} are symmetric. To make writing easier, we shall always assume we are between stages  \ref{current_birth} and \ref{finsym1} when anything relevant happens there.

    Let us first assume that we hit stage \ref{cuts1} in Algorithm \ref{cuts}. Then we have hit stage \ref{Cdeath1}. By Property \ref{facts}, the time of death $x_d(l_2)$ has decreased. Since the times of death are bounded from below \eqref{late_death} and may take only a finite number of values, that of crossroads, we shall hit stage \ref{cuts1} only a finite number of times.

    Now, if we do not hit stage \ref{cuts1}, the variable $Cuts$ stays at zero at the end of the cutting loop, so that Algorithm \ref{main} ends.

    Moreover, in this case,  we do not hit \ref{Cdeath1}, so there is no change to the pretessellation during the last cutting loop. So that the final pretessellation $P_o$ is the same as the one at the end of the last parent-seeking loop. So that by Property \ref{facts}, $P_o$  is a pretessellation~\eqref{nocross} with late births~\eqref{late_birth} and late deaths~\eqref{late_death}, with births occurring on the relint of segments or on the boundary~\eqref{born_on_segment}, and before the true death times~\eqref{superposition}. The births that happen before the real death of the parent are real \eqref{realchild}.
    
    Furthermore, each line in $P_o$ has the same number  of children as in $T$: if a line $l$ had more, it would pass the conditional stage \ref{current_birth} in Algorithm \ref{cuts} at least $(O^T(l) + 1)$ times, and thus pass stage \ref{trop} and hit stage \ref{cuts1}. On the other hand, the total number of children of lines in $P_o$ is at least as much as in $T$: each line is born exactly once, and if it is born on the boundary in $P_o$, it is also in $T$ since times of birth are overestimated.
    
    Finally, if a line $l$ has the same time of death as in $T$, then all its children are born before its true time of death, so that they are all real children. Moreover $l$ has the right number of children, so it has all its true children.
\end{proof}

The lemma states that the algorithm ends, but not that we have the real tessellation. I confess that I do not know whether we may have pathological situations where the same input satisfies the requirements with respect to several different pretessellations.
However, we now circumvent the difficulty by carefully choosing the set $U_0$ of orphan lines.

For a given set $U_0$ satisfying requirement~\eqref{binary}, we may write $P_o$ or $P_o(U_0)$ for the output pretessellation. We also call $D_b(U_0)$ the set of lines with wrong times of birth, that is $x_b^{P_o}(l) \neq x_b^T(l)$. Obviously $D_b(U_0) \subset U_0$.

Now, if we choose the right line and give its birth time, then there will be at least two less lines in $D_b$: that one and another. 
Formally:

\begin{lem}
\label{subU}
Let $U_0$, $P_o$ and $D_b(U_0)$ defined as above.

Then there is a line $l$ such that 
\begin{align}
\label{reduc}
\# D_b(U_0 \backslash \left\{ l \right\} ) \leq \# D_b(U_0) - 2.
\end{align}
\end{lem}

\begin{proof}
\begin{figure}
\subfloat[][]{
             \label{rustine1}
             \includegraphics[width=.49\textwidth]{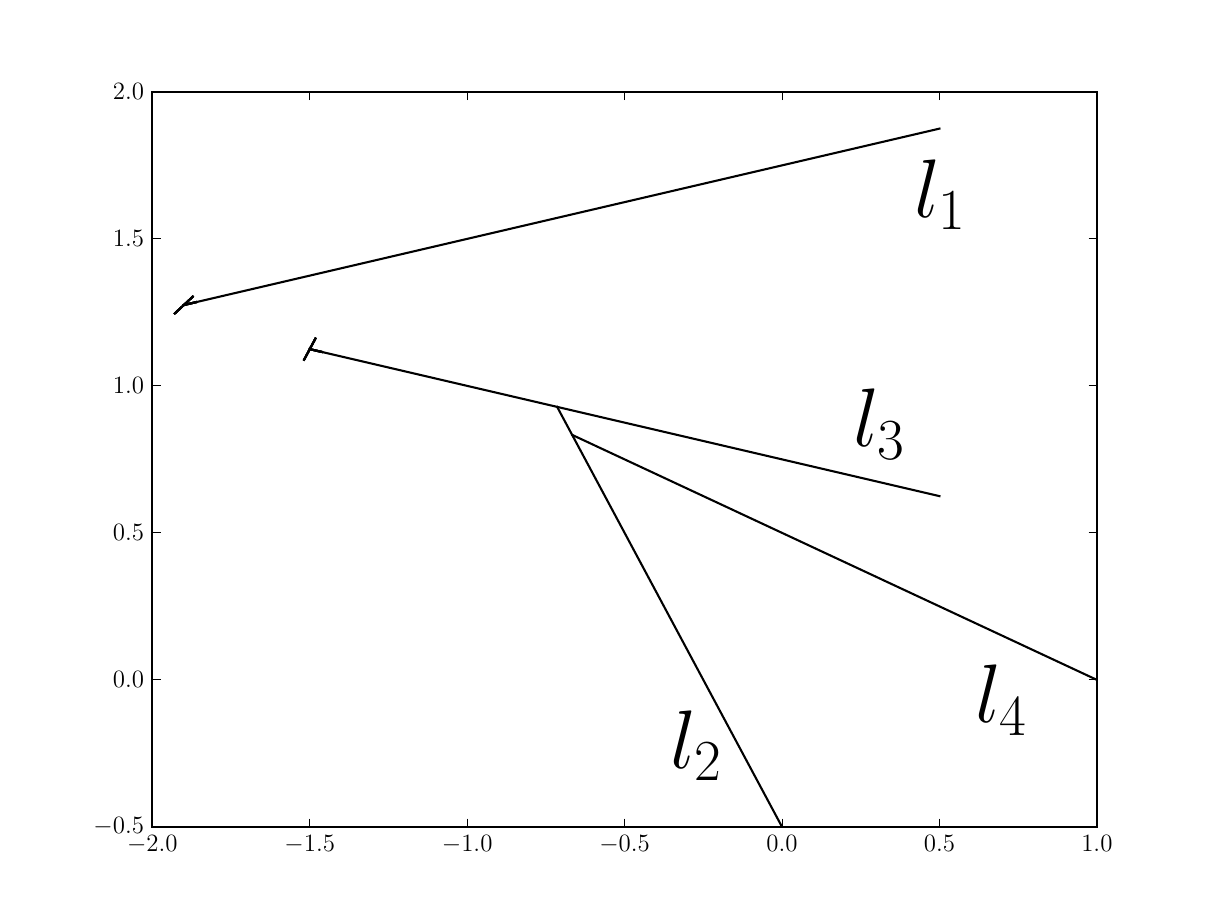}
             }
\subfloat[][]{
             \label{rustine2}
             \includegraphics[width=.49\textwidth]{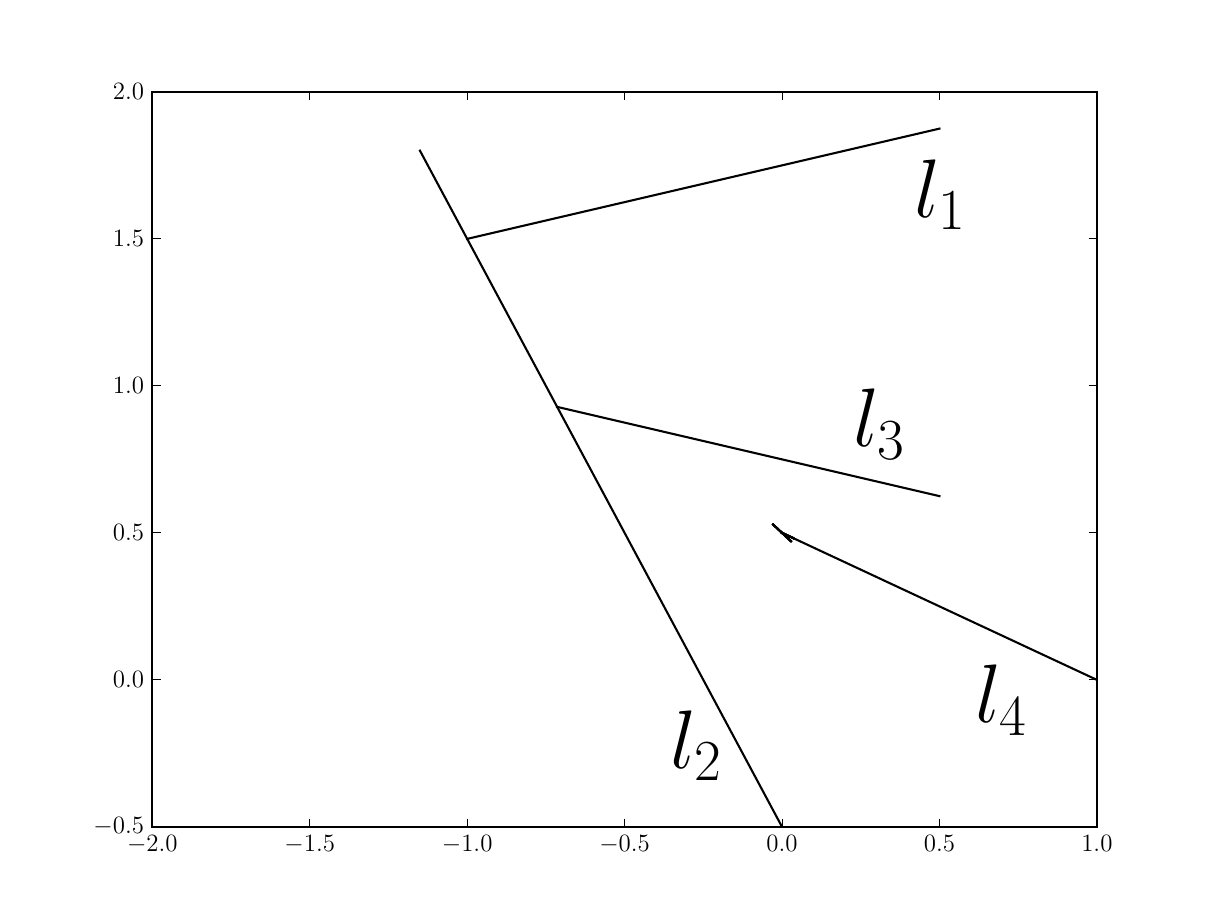}
             }\\
\subfloat[][]{
             \label{rustine3}
             \includegraphics[width=.49\textwidth]{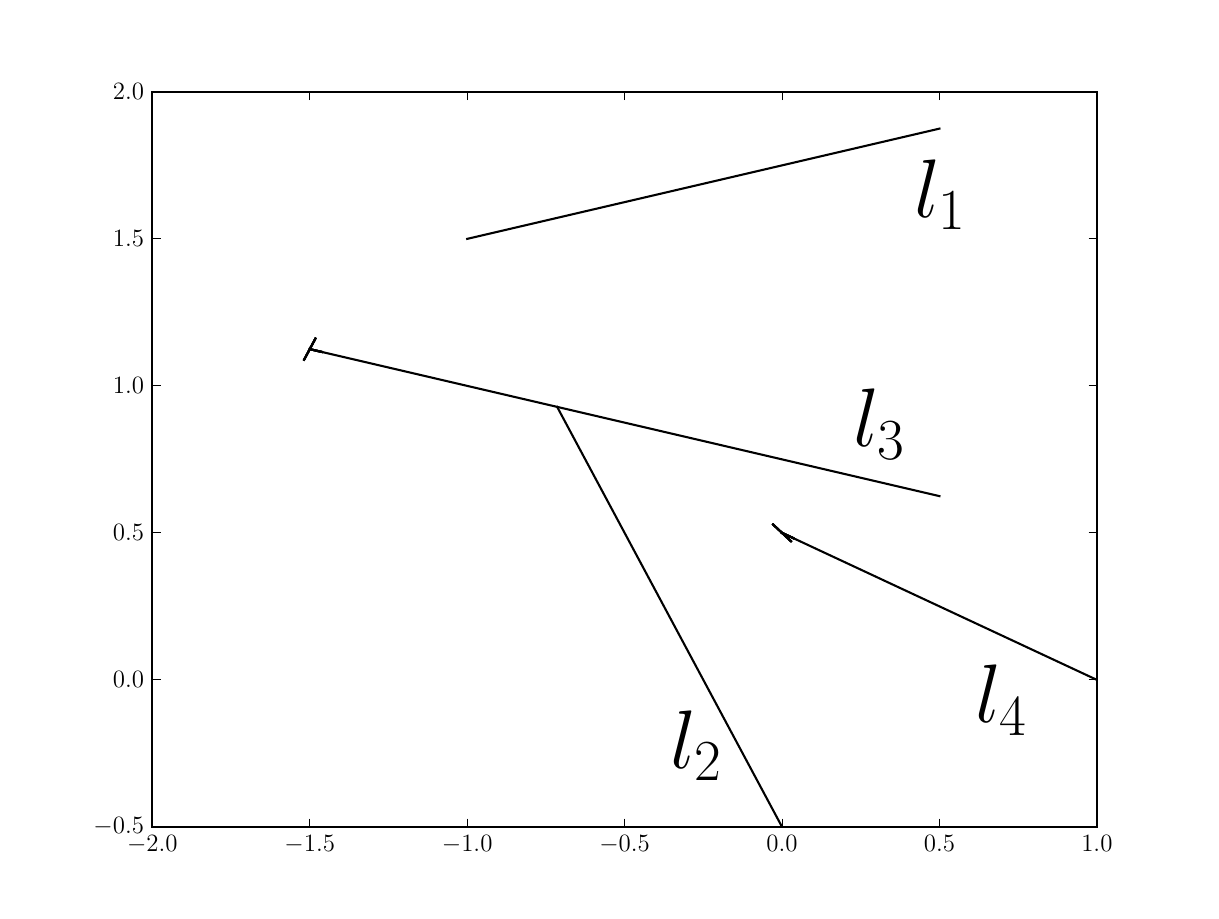}
             }
\subfloat[][]{
             \label{rustine4}
             \includegraphics[width=.49\textwidth]{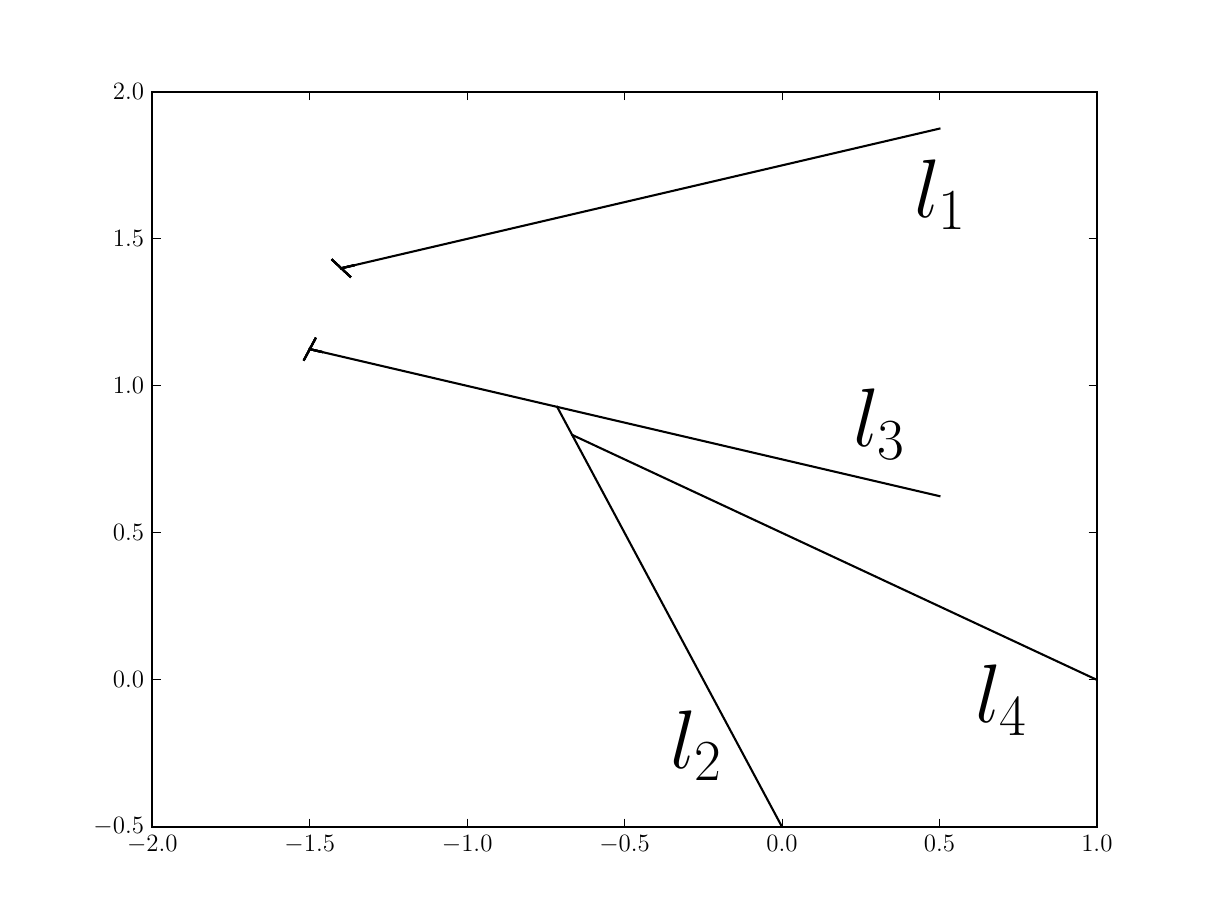}
             }
             \caption{The real segments on the four lines are sketched in \protect\subref{rustine1}. Suppose that at the end of Algorithm \ref{main}, we instead obtain \protect\subref{rustine2}, with three wrong birth times for lines $l_1$, $l_3$, and $l_4$. If we now rerun the algorithm with the  birth time of $l_3$ specified \protect\subref{rustine3}, the line $l_2$ will need another child and the birth time of $l_4$ will be set right. Notice that $l_1$ is also born sooner, but not necessarily soon enough yet \protect\subref{rustine4}. So that we correct at least two birth times by specifying a single good one.}
\label{rustine}
\end{figure}
A first remark is that for any subset $U_1$ of $U_0$, Lemma \ref{ends} holds with $P_o(U_1)$ instead of $T$ and $P_o(U_0)$ instead of $P_o$. Indeed there would be no change when running Algorithm~\ref{main} with $U_0$ as input if $P_o(U_1)$ was the real tessellation. So that $x_b^{P(U_0)}(l) \geq x_b^{P(U_1)}(l) \geq x_b^T(l)$ for all lines. This implies that $D_b(U_1)$ is a subset of $D_b(U_0)$.

Figure \ref{rustine} illustrates how we choose which line to remove from $U_0$.

We consider $l_1 \in D_b(U_0)$ and its fake parent $l_2$, so that $x_b^{P_o}(l_1) = \crossroad(l_2, l_1)$. Since births that happen before the real death of the parent are real \eqref{realchild}, we know that   $\crossroad(l_2, l_1) > x_d^T(l_2)$.
The line $l_2$ is killed in the real tessellation by a line $l_3$, that is $x_d^T(l_2) = \crossroad(l_2, l_3)$. Notice that $l_3$ and $l_1$ may be the same line. We shall correct the time of birth of this line $l_3$, the line $l$ in the lemma.

This notably entails that $l_1$ is no longer a child of $l_2$ in  $P( (U_0 \backslash \left\{ l_3 \right\} )$. But the number of children of $l_2$ at the end of the algorithm is fixed, equal to that in the real tessellation $T$. So that there is a line $l_4$ that is a child of $l_2$ in $P( (U_0 \backslash \left\{ l_3 \right\} )$ and was not in $P(U_0)$. Since moreover $l_2$ is killed by the right line in $P (U_0 \backslash \left\{ l_3 \right\} )$, we know that $l_4 \in D_b(U_0)$ and $l_4 \not \in D_b(U_0 \backslash \left\{ l_3 \right\})$.

Finally, since $D_b(U_0 \backslash \left\{ l_3 \right\} ) \subset D_b(U_0)$, we may write:
\begin{align*}
\# D_b(U_0 \backslash \left\{ l \right\} ) & \leq \# D_b(U_0) - \#\left\{ l_3, l_4 \right\} 
=  \# D_b(U_0) - 2.
\end{align*}
\end{proof}

We may now state and prove:
\begin{thm}
    \label{final}
    There is an input $(U_0, x_b^P, O^T, V)$ such that:
\begin{itemize}
    \item{Output of Algorithm \ref{main} is $T$.}
\item{The set of orphan lines is at least as big as one fourth of the number of internal nodes of the tree of births, except the boundary:
    \begin{align*}
        y & \hat{=}  \#\! \left\{ m \in L: \exists l \in L, x_b^T(l) = \crossroad(l,m)   \right\} ,\\
        \# U_0 & \geq \frac{y}4 .   
\end{align*}
}
\end{itemize}

    This entails Theorem \ref{combi}.
\end{thm}

\begin{proof}
We first build a big set of orphan lines $U_0$ satisfying requirement~\eqref{binary}, then use Lemma \ref{reduc} to build the $U_0$ promised in the theorem.

Start by  looking at the tree of births. We define $U_0$ as the set of inner nodes of even or odd generations, except the root, whichever is the biggest. This ensures that $U_0$ has cardinal at least $u_0 = y/2$. Moreover, since each line in $U_0$ is an interior node in the tree of lifes, it has at least a child. Since the children's generations have opposite parity to their parents', none is in $U_0$, and requirement~\eqref{binary} is satisfied.

We may then run the algorithm and find an output pretessellation $P_o(U_0)$.  It might not be the true tessellation, as there might be a set of lines $D_b(U_0)$ whose parent is wrong. This set is included in $U_0$. We then remove a line from $U_0$ as in Lemma \ref{subU} and run the algorithm again. And we iterate until we obtain the true tessellation. Since $D_b(U_0)$ is at least two elements smaller at each step, we have to remove at most $u_0/2$ lines from our initial $U_0$ to get a set $U_0$ that we may use as input in Algorithm~\ref{main} to obtain a pretessellation $P_o$ with all birth times right, that is $D_b(U_0) = \emptyset$. So that the cardinal of the final $U_0$ is at least $u = y / 4$.

This set of orphan lines yields the true tessellation $T$. Indeed, since all the times of birth are right  and $P_o$ has late deaths \eqref{late_death},  the segments of $P_o$ contain those of $T$. Since $P_o$ is a pretessellation, segments do not cross, hence the times of death cannot be later than in the real tessellation $T$.

\medskip

We now prove Theorem \ref{combi}.

Let us fix the set of interior nodes of the tree of births, except the boundary. Since it is a subset of $L$ of cardinal $k$, by Example \ref{ex1}, there are at most $2^k$ possibilities.

With this set fixed, we now bound the number of different values each element of the input may take, while also satisfying the requirements in Theorem \ref{final}:
\begin{itemize}
    \item{$U_0$ is a subset of a set with $k$ elements, so by Example \ref{ex1}, there are at most $2^k$ different possible $U_0$.}
    \item{$O^T$ and $V$ are functions $f$ from a set with $k+1$ elements to the natural numbers, such that $\sum_l f(l) \leq k$, so by Example \ref{ex2}, there are at most $4^k$ possibilities for each.}
    \item{$x_b^P$ is a function from $L \setminus U_0$, where each image is from a set of cardinal $y+1$. Indeed $x_b^P(l) = \crossroad(l,m)$ for $m$ the parent of $l$, so that $m$ is either the boundary or an interior node of the tree of births. Hence there are at most $(1 + y)^{k - y / 4}  $ different possible $x_b^P$.}
\end{itemize}

Thus we may give the following upper bound on the number of different T-tessellations on $k$ given lines, using $C$ for any  constant:

\begin{align*}
\mathcal{N} (k) & \leq C^k \sup_{0 \leq y \leq k-1}  (1 + y)^{k - y / 4} \\
& \leq C^k \left(\frac{k}{(\ln k)^{1-\epsilon }}\right)^{k - k / (\ln k)},
\end{align*}
where we have used the following bound on the supremum in the right-hand side: take the derivative in $y$ of the logarithm, and we see that the maximum is attained when
\[
(1 + y) (1 + \ln(1 + y)) = 4 k - 1.
\]
For big $k$, this implies $k / \ln(k)^{1 - \epsilon} \geq 1 + y \geq 4k / \ln(k)$. We then replace by the right bounds in the exponent and the basis.
\end{proof}

\section{Optimality remarks and perspectives}
\label{perspectives}

Though we have used very violent upper bounds at times, there is no way to get a substantially better combinatorial result. Indeed let us consider the following $k$ lines on a square domain $[0,1]^2$, for some integer $a \leq k$:
\begin{align*}
y & = \frac{\lambda}{a+1}        & \mbox{for $\lambda \in [1, a]$} \\
x & = \frac{\lambda }{k - a + 1} & \mbox{for $\lambda \in [1, k -a ]$.} 
\end{align*}
How many different T-tessellations can we build on those lines? A lower bound is given by supposing that all horizontal segments are maximal, that is have both endpoints on the boundary. Then each of the vertical segments is between two consecutive horizontal lines, and hence of length $1/(k-a+1)$. More significantly, this means each one can be at $(k - a + 1)$ different places, independently from each other since the vertical lines do not cross. So that there are at least $(k-a+1)^a$ different T-tessellations that can be built on those lines. If we take $a = k - k / (\ln k)$, we may conclude:
\begin{lem}
\label{upper}
There are sets of $k$ lines such that the number of T-tessellations on those lines admits the following lower bound:
\begin{align*}
\mathcal{N} (k) \geq \left(  \frac{k}{\ln k}  \right) ^{k - k / (\ln k)}
\end{align*}
\end{lem}

If we want to get a better result and a tighter upper bound on the partition function, we then need to have a closer look on the usual topologies of the lines. That is an order of magnitude harder, but might be worth the effort.
Indeed the previous worst-case example hinges heavily on having many lines crossing many segments, and topologically equivalent sets of lines have very low measure, looking like $(k 2^k k!)/(2k)!$ of the space of all sets of $k$ lines. 

By contrast, using very sloppy heuristics, we would expect that for most sets of $k$ lines, the number of T-tessellations on those lines behaves like
\[
\mathcal{N} = \sqrt{k}^k.
\]
The idea is the following: let us take a segment away of the true tessellation. How many different segments may we put on the line to get a tessellation again? Neglecting problems of children and murders, this would be the number of segments that the line cross, plus one. Now the probability of crossing a segment is essentially the length of this segment. So the number of crossed segments would be $k m$, where $m$ is the mean length of a segment. Now the mean length of a segment is the mean interval between two successive segments a line cross, that is $1/(km)$. So that $m$ should be of order $1/\sqrt{k}$, and for each new line, we have $k / \sqrt{k} = \sqrt{k}$ as many possibilities.

Thus it seems likely that the method in this paper gives little information on the measure, except its very existence.

\section*{Acknowledgements}
\label{ack}

I would like to thank the referees, whose many remarks have greatly contributed to make the article more readable.

\bibliographystyle{plainnat}
\bibliography{tessellations}

\appendix
\section{Technical proofs}
\label{proofs}

\subsection{Proof of Lemma \ref{distinct}}
\label{pdistinct}

\begin{proof}
    Denote by $l(\varepsilon ) = l + \varepsilon \vec{u}_{x}$ the translate of $l\in L$ by the translation of $\varepsilon $ along the time axis. Then for any other line $m \in L$, either $m$ is parallel to $l$, and $\crossroad(l(\varepsilon ),m)$ stays at infinity for $\varepsilon $ small enough, or $\crossroad(l(\varepsilon ),  m)$ is a continuous function of $\varepsilon$. Both $ \crossroad(l(\varepsilon ),  \boundary)$ and $ \crossroad(\boundary, l(\varepsilon ))$ are also  continuous function of $\varepsilon$ for $\varepsilon $ small enough. Since $l$ is not colinear to the time axis, all those functions are injective.

    So that for $\varepsilon $ small enough, strict order is preserved: if we write $m' = m(\varepsilon )$ if $m = l$ and $m' = m$ if $m\in (L\setminus \{l\})\cup\boundary$, then $\forall l_1, m_1, l_2, m_2 \in L\cup\boundary, \,\,\,\,\crossroad(l_1, m_1) >   \crossroad(l_2, m_2) \implies  \crossroad(l_1', m_1') > \crossroad(l_2', m_2')$  .
 
    Moreover there are finitely many crossroads, so that except for a finite number of $\varepsilon $, the translated line has all its crossroads distinct from any other crossroad: $ \forall  m_1, m_2, m_3 \in (L\setminus \{l\}) \cup\boundary, \crossroad(l(\varepsilon ), m_1) \neq \crossroad(m_2, m_3) $ and $\crossroad(\boundary, l(\varepsilon )) \neq \crossroad(m_2, m_3)  $ as points in the plane.

Hence we may map all lines $l$ in turn to $l' =  l + \varepsilon \vec{u}_x $, where $\varepsilon $ may depend on $l$ and is small enough, and we get a set of $k$ lines $L'$ with all crossroads distinct and such that strict order is preserved: 
\begin{align}
    \label{strictorder}
    \crossroad(l_1, m_1) >   \crossroad(l_2, m_2) & \implies  \crossroad(l_1', m_1') > \crossroad(l_2', m_2') & \forall l_1, m_1, l_2, m_2 \in L\cup\boundary.
\end{align}

\medskip

Let $T$ be a T-tessellation on $L$. We define $T'$ on $L'$ by mapping the trees of birth and death: if $l$ is a child of $m_1$ in $T$ and is killed by $m_2$ in $T$, then $l'$ is a child of $m'_1$ and is killed by $m'_2$ in $T'$, so that $s(l') = [\crossroad(m'_1,l') ,\crossroad(l',m'_2) ]$. Since strict order is preserved, death does happen after birth, and segments $s(l')$ are well-defined. Moreover, since strict order is preserved, $\crossroad(m'_1,l') $  is in the relint of $s(m'_1)$ and $\crossroad(l',m'_2) $ is in the relint of $s(m'_2)$  (with the convention $s(\boundary) = \boundary$). Let us prove $T'$ is really a T-tessellation.

All crossroads in $T'$ are distinct, so vertices can only involve the two elements of $L\cup \boundary$ defining the corresponding crossroad  $\crossroad(l',m')$. If it is of degree at least $3$, then both $s(l)$ and $s(m)$ (with the convention $s(\boundary) = \boundary$) include $\crossroad(l,m)$ in $T$. So that, one is born or killed by the other in $T$, say $m$ is born on $l$. By definition of $T'$, then $m'$ is born on $l'$, and $\crossroad(l',m') $ is a T-vertex in $T'$. Conversely, if say $s(l')$ has only one edge incident to $\crossroad(l',m') $, then by definition of $T'$, $l'$ is born on or killed by $m'$, and  $\crossroad(l',m') $ is again a T-vertex in $T'$. So that $T'$ is indeed a T-tessellation.

Moreover, the procedure is injective: since all crossroads in $T'$ are distinct, different trees of birth and death yield different T-tessellations. So that   $ \# \mathcal{T} (L) \leq \# \mathcal{T} (L')$.
\end{proof}

\subsection{Proof of Lemma \ref{lemalgo1}}

At the end of initialisation, we have the following properties:
\begin{itemize}
\item{Birth times are those of the tessellation for all lines $l$: $x_b^P(l) = x_b^T(l)$.}
\item{Death times are overestimated for all lines $l$: $x_d^P(l) \geq x_d^T(l)$.}
\item{The number of remaining murders $M(l)$ for each line is that of the true tessellation.}
\end{itemize}

Indeed the first and third points are merely the input, and the death times are set to an upper bound at stage \ref{bigdeath}.

\bigskip

What is important is that those properties will remain true throughout the loop that completes the algorithm (stage \ref{a1loop} of Algorithm \ref{algo1}). This will yield by recurrence that at the end of the $\crossroad$ iteration of the loop:
\begin{itemize}
\item{The remaining number of murders for each line $M(l)$ is that of the true tessellation $M^T(l, \crossroad) = \#\left\{ m \in L: x_d^T(m) = \crossroad(m,l) \mbox{ and } \crossroad(m,l) > \crossroad \right\} $.}
\item{Death times before $\crossroad$ are right, that is: $(x_d^T(l) \leq \crossroad) \Rightarrow (x_d^P(l) = x_d^T(l))$}
\end{itemize}

We have to prove that if this is true before the $\crossroad(l_1, l_2)$ iteration, it will be true after it.

\bigskip

Now, we pass the condition on stage \ref{a1cond1} if and only if there is a death in the real tessellation. Indeed, in that case
\[
x_b^P(l_1) = x_b^T(l_1) < \crossroad(l_1, l_2) \leq x_d^T(l_1) \leq x_d^P(l_1),
\]
and the same for $l_2$. If on the contrary there is no death, since segments do not cross \eqref{nocross}, either $x_b^T(l) \geq \crossroad(l_1, l_2)$ for one of the two lines, and then this also holds for $x_b^P(l) = x_d^T(l)$, or one of the two lines $l$ is already dead $x_d^T(l) < \crossroad(l_1, l_2)$. Then $x_d^P(l) < x_d^T(l)$ by recurrence hypothesis.

If we do not pass the condition, there are no changes to $M$ or $x_d^P$. On the other hand, there is no change to $M^T$, nor any new line whose death time is required to be right, so the conditions still hold.

If we do pass the condition, then either $x_d^T(l_1) = \crossroad(l_1, l_2)$,\linebreak[4] or $x_d^T(l_2) = \crossroad(l_1, l_2)$. In the first case, using the recurrence hypothesis,\linebreak[3] \mbox{$M(l_1) = M^T(l_1, \crossroad(l_1, l_2)) = 0$}, and $x_d^P(l_1)$ is set to $x_d^T(l_1)$, satisfying the second condition. The first condition is also still satisfied, since the only number of murders that changes for the real tessellation is that of $l_2$, which decreases by one, since $l_1$ is no more in the set of \emph{remaining} murders. Symmetrically, if $x_d^T(l_2) = \crossroad(l_1, l_2)$, then $M(l_1) = M^T(l_1, \crossroad(l_1, l_2)) > 0$ since it contains $l_2$, and this number of remaining murders is decreased by one while $x_d^P(l_2)$ is set to $x_d^T(l_2)$. So that the recurrence hypothesis is transmitted.

\bigskip

Since death times before $\crossroad$ are right, after we hit the last crossroad, all death times  are right, that is $x_d^P(l) = x_d^T(l)$ for all lines $l$. Hence the output pretessellation is the real tessellation.

\subsection{Proof of Proposition \ref{facts}}

The proof is implicitly a recurrence, following the algorithm. We may use the properties to be proved in the proof itself, with the intended meaning that they hold till that point.

First, the fact that birth before the true time of death of the parent is real~\eqref{realchild} is a consequence of the other properties. Thus we won't have to check it separately. Indeed, since births are late~\eqref{late_birth} but sooner than the true death~\eqref{superposition}, the time of birth is included in the true segment, maybe as first point: $x_b^T(m) \leq x_b^P(m) = \crossroad(l,m) < x_d^T(m)$. Since deaths are late, the LHS of equation~\eqref{realchild} means that we are in the relative interior of the parent $l$. Since there is no crossing~\eqref{nocross}, the line $m$ is really the child of $l$, that is $ x_b^T(m) = x_b^P(m)$.

\bigskip

A few conditions already hold after preinitialisation. 

Indeed on the one hand at stage \ref{death1} of Algorithm \ref{main}, we set the death times of all lines to the maximum possible, that is the rightmost point of the domain. So that $x_d^P(l) \geq x_d^T(l)$, deaths are late~\eqref{late_death}.

On the other hand, we know the birth times of the lines $l$ not in $U_0$. For those, $x_b^P(l) = x_b^T(l) < x_d^T(l)$. For the lines $l \in U_0$, whose parent we do not know, we set their birth time to that of their first child at stage \ref{birth1}. Child births are on the relative interior of a segment~\eqref{born_on_segment}. So that $x_d^T(l) > x_b^P(l) > x_b^T(l)$. So that births are late~\eqref{late_birth}, but before true death~\eqref{superposition}. Moreover, since $U$ is initialised as $U_0$, condition~\eqref{late_unknown} is fulfilled.

\medskip

Suppose we prove that up to a point in the algorithm,  birth and death times are decreasing~\eqref{decrease}. Then birth will still occur before true death~\eqref{superposition}.
 We thus do not check those facts separately.

 \medskip

A first remark is that times of birth and death are changed only at stages  \ref{birth2}, \ref{birth3},  \ref{death2}, \ref{death3} and \ref{death4} during initialisation. In all cases, there is a conditional stage just before requiring that the former value be greater, that is $x_{\bullet}(l) > \crossroad$. So that times of death and birth decrease throughout initialisation. The same is true within Algorithm~\ref{parent}, with changes at stages \ref{Parent_birth1} and \ref{Parent_birth2}. Thus, we only have to check property \eqref{decrease} in Algorithm~\ref{cuts}.

\bigskip

We shall now prove that the prototessellation at the end of initialisation satisfies the conditions in Property~\ref{facts}. It is now enough to prove that it is exactly the pretessellation described as $P(U_0)$ in equations \eqref{birth_nU0}, \eqref{birthU0} and~\eqref{deathU0}
.

\medskip

We use recurrence. We are following the $\algorithmicfor$ loop~\ref{forloop}. Iterations follow the crossroads timewise. With the following definition, 
\[\mathcal{V} (l, \crossroad) = \left\{ m \in L :  \crossroad \leq \crossroad(l,m) = x_d^{P(U_0)}(m) < x_d^{P(U_0)}(l),  \right\},\]
the recurrence hypothesis is, at the start of the $\crossroad(l_1, l_2)$ iteration:
\begin{align*}
    x_b^P(l) & = x_b^T(l) = x_b^{P(U_0)}(l)   &\text{for all $l \not \in U_0$,} \\
    x_b^P(l) & = 1        & \text{for all $l \in U_0$ such that $x_b^{P(U_0)}(l) \geq \crossroad(l_1, l_2)$,} \\
    x_b^P(l) & = x_b^{P(U_0)}(l)& \text{for all $l \in U_0$ such that $x_b^{P(U_0)}(l) < \crossroad(l_1, l_2)$,} \\
    x_d^P(l) & = 1     &  \text{for all $l$ such that $x_d^{P(U_0)}(l) \geq \crossroad(l_1, l_2)$,}  \\
    x_d^P(l) & = x_d^{P(U_0)}(l)   &  \text{for all $l$ such that $x_d^{P(U_0)}(l) <\crossroad(l_1, l_2)$,}  \\
    V^P(l)      & = \# \mathcal{V} (l, \crossroad(l_1, l_2)) & \text{for all $l \in L \cup \boundary$.}
\end{align*}
The recurrence hypothesis is satisfied after preinitialisation, and entails that $P=P(U_0)$ at the end of the $\algorithmicfor$ loop (we may add a do-nothing fictitious event at time $1$ to see the effect of the last iteration).

\emph{Transmission of recurrence hypothesis} 

In the  $\crossroad(l_1, l_2)$ iteration, functions may change only on $l_1$ or $l_2$. Now, since $x_{\bullet}^{P(U_0)}(l)$ may only take $\crossroad(l,\bullet)$ as a value, the equalities will remain valid for all the other lines $l \in L$. Same thing for $V^P$.

To hit stage~\ref{birth2}, we need $x_b^P(l_1) =  \crossroad(l_1, l_2)$. By recurrence, this means that $l_1 \not \in U_0$. So that it is really born here. Since segments are born  on the relative interior of their parent~\eqref{born_on_segment}, we get that $l_2 \in U_0$ and 
hence that $ x_b^{P(U_0)}(l_2)\geq \crossroad(l_1, l_2) $. So that $l_2$ has no real child before that time $ (\crossroad(l_2,m) = x_b^T(m) ) \implies \crossroad(l_2,m) \geq \crossroad(l_1, l_2)$. On the other hand, $l_1$ is its child, so that we have in fact $ x_b^{P(U_0)}(l_2) = \crossroad(l_1, l_2)$. Now stage~\ref{birth2} sets $x_b^P(l_2)$ to this value. Since it was $1$ beforehand, it has decreased~\eqref{decrease}.

Conversely, if  $ x_b^{P(U_0)}(l_2) = \crossroad(l_1, l_2)$, then $l_1 \not \in U_0$ and, by recurrence, we hit stage~\ref{birth2} and set $x_b^P(l_2)$ to $x_b^{P(U_0)}(l_2)$ 

Stage~\ref{birth3} is the same, switching $l_1$ and $l_2$. So that the recurrence hypotheses on birth times is transmitted.

\smallskip

By recurrence, we pass the $\algorithmicif$ statement~\ref{vmurder} if and only if $ x_b^{P(U_0)}(l) < \crossroad(l_1, l_2)  $ and $ x_d^{P(U_0)}(l) \geq \crossroad(l_1, l_2)  $ for both $l \in \{l_1, l_2\}$.

By recurrence, $V^P(l_1) = 0$ if  $ x_d^{P(U_0)}(l_1) = \crossroad(l_1, l_2) $. So that stage~\ref{death2} is attained, and $x_d^P(l_1) \gets\crossroad(l_1, l_2)  $. Since beforehand $x_d^{P(U_0)} = 1$ it has decreased as expected~\eqref{decrease}. If moreover  $ x_d^{P(U_0)}(l_2) \neq \crossroad(l_1, l_2) $, then 
$  l_1 \in \mathcal{V}  (l_2, \crossroad(l_1, l_2))$, so that $V^P(l_2) > 0$. Hence its death time does not change, and stage~\ref{vkill} is attained. At the start of the next iteration, at crossroad $\crossroad$, the line $l_1$ will not be anymore in $ \mathcal{V} (l_2, \crossroad)$, so that the recurrence hypothesis is indeed transmitted.

Switching the role of $l_1$ and $l_2$, we may reason in the same way.
Since there is no crossing~\eqref{nocross} in $P(U_0)$, at least one of  $ x_d^{P(U_0)}(l_1) $ or $  x_d^{P(U_0)}(l_2) $  is $ \crossroad(l_1, l_2) $, so that all cases are covered.

\emph{End of transmission of recurrence hypothesis}

\bigskip

We now have to deal with the main loop, separated in functions $Parent\_seek$ and $Cutting$. 

\medskip

Within $Parent\_seek$, the changes occur at stages \ref{Parent_birth1} and \ref{removeU1}, and symmetrically  \ref{Parent_birth2} and \ref{removeU2}. We shall always assume that we hit stage~\ref{Parent_birth1} instead of stage~\ref{Parent_birth2}. Since any line whose birth time we change is excluded from $U$, condition~\eqref{late_unknown} still holds. There are no change in death times, so that condition~\eqref{late_death} still holds. We already mentioned that condition~\eqref{decrease} hold throughout $Parent\_seek$.

To prove that all the lines are born on relative interiors of segments or on the boundary, we shall need this result:
\begin{property}
    \label{tech}
    Whenever we enter Algorithm $3$, if a line $l$ is born at crossroad $\crossroad(m,l)$, then at least one of the following conditions holds:
    \begin{itemize}
        \item{$m$ is the boundary $\boundary$.}
        \item{$m\in U$ and $x_b^P(m) = \crossroad(m,l) $.}
\item{$l\in U$.}
\item{We are on the relative interior of the parent segment: $x_b^P(m) < \crossroad(m,l) < x_d^P(m)$}
\end{itemize}
\end{property}

Property~\ref{tech} holds in $P(U_0)$: if $l\not \in U_0 = U$, then its birth time is right, that is $ x_b^{P(U_0)}(l) = x_b^T(l)$. Now if its parent $m$ is not in $U_0$, then we are on the relative interior: $ x_b^P(m) = x_b^T(m) < x_b^T(l) = x_b^P(l) < x_d^T(m) \leq x_d^P(m) $. If its parent $m$ is in $U_0$, then either it is its first child and  $x_b^P(m) = \crossroad(m,l) $, or $ x_b^P(m) < x_b^T(l) = x_b^P(l) < x_d^T(m)  \leq x_d^P(m) $. Later on, we will have to check this holds after Algorithm~\ref{cuts}. 

\smallskip

Now, we shall prove that all the lines in $U$ hit stage~\ref{Parent_birth1} exactly once. 

To start with, they cannot hit it more than once, since they are excluded from $U$ afterwards and will not pass the $\algorithmicif$ stage~\ref{cond1}. On the other hand, if they have not passed stage~\ref{cond1} earlier in the loop, they will at their true birth time $x_b^T(l)$. Since $l\in U$, we know that $x_b^P(l) > x_b^T(l)$. Besides, either the true parent $m$ of $l$ is not in $U_0$, or $l$ is not its first child. In both cases, $x_b^P(m) \leq x_b^{P(U_0)}(m) < x_b^T(l) = \crossroad(m,l) < x_b^T(m) \leq x_b^P(m)$, so that stage~\ref{cond1} is passed. We have used that birth and death times are late, and that changes decrease those times.

Since the loop is reverse timewise, statement~\ref{Parent_birth1} will be hit when we reach $x_b^T(l)$ at the latest, ensuring property~\eqref{late_birth}. Let us turn to being born  on the relative interior of its parent or on the boundary~\eqref{born_on_segment}. The form of the conditional stage~\ref{cond1} guarantees the property if $l\in U$. If $l\not\in U$, by Property~\ref{tech}, either it was already satisfying~\eqref{born_on_segment} before the loop, or
it was born on $m\in U$ with  $x_b^P(m) = \crossroad(m,l) $. Since $m$ is strictly prolongated backwards, $l$ will be born on the relative interior of $m$ after the loop.

Finally, we must show that $P$ stays a pretessellation, that is that a prolongated segment $s(l)$ does not cross any other segment, \emph{i.e.} that no point of the prolongation $\pi(l) = (x_b^{P_{after}}(l), x_b^{P_{before}}(l)]$ is in the relative interior of any segment. On the one hand, those prolongations are included in the true segment by properties~\eqref{superposition} and~\eqref{late_birth}. So that no two prolongations can cross.
On the other hand, the prolongation cannot cross the before-the-loop pretesselation $P_{before}$. If $x_b^{P_{before}}(m) < \crossroad(m, l) < x_d^{P_{before}}(m),x_b^{P_{before}}(l)  $ when we are still prolongating $l$, then we pass condition~\ref{cond1} and $ \crossroad(m, l) = x_b^{P_{after}}(l)$. So that segments still do not cross~\eqref{nocross}.

\bigskip

Within $Cutting$, the changes happen at stages \ref{Cdeath1}, \ref{cuts1}, \ref{U1} and the symmetrical   \ref{Cdeath2}, \ref{cuts2}, \ref{U2}. We shall always assume we hit stages~ \ref{Cdeath1}, \ref{cuts1} and \ref{U1} instead of their symmetric stages. No birth time is changed so births stay late \eqref{late_birth}.

To understand what is going on, let us consider a crossroad $\crossroad(l,m)$ where $l \in U_0$, $m\in L$, and $x_b^P(l) = \crossroad(l,m)$ at input. Since the birth time is included in the segment \eqref{born_on_segment}, if the death time $x_d^P(m)$ is set at stage~\ref{Cdeath1}, it is decreased, ending the proof of property~\eqref{decrease}.



The counter $O$ is the number of $U$-children of $m$ at input that are born at $\crossroad(l,m)$ at the latest. So that if we change its death time (stage \ref{Cdeath1}), there are $O^T(m)$ such children strictly before $\crossroad(l,m)$. Since it cannot have more children, and all children before its true death are real~\eqref{realchild}, we obtain $x_d^T(m) \leq \crossroad(l,m)$ and the condition on death times \eqref{late_death} is still fulfilled. The bound $x_d^T(m) \leq \crossroad(l,m)$ holds for the same reason if we hit stage~\ref{U1}, so that $m$ is not the true parent of $l$ and we have $x_b^P(l) < x_b^T(l)$. Thus condition~\eqref{late_unknown} is still fulfilled.

Notice that $O(\boundary) > O^T(\boundary)$ is impossible thanks to the late births~\eqref{late_birth}.

Since segments are only shortened during $Cutting$, they will not cross and $P$ stays a pretessellation.

We now only have to check that Property~\ref{tech} holds at the end of $Cutting$. The only possible problem would be from children of a line $m$ that has been cut, born between $x_d^{P_{after}}(m)$ and $x_d^{P_{before}}(m)$. But all those lines are added to $U$ at stage~\ref{U1}.
 This ends the proof.

\section{Algorithms and corresponding figures}
\label{alg_and_fig}

\begin{algorithm}
\caption{Rebuild from tree of births and number of murders.}
\label{algo1}
\ \\
\noindent
\textbf{Input:} \ The set $L$ of lines of the tessellation, a prototessellation $(x_b, x_d)$ such that $x_b = x_b^T$,  a murder function $M: L \cup \boundary \to \mathbb{N}$ such that $M = M^T$ the number of murders in the real tessellation \eqref{MT}, the ordered set $\oCrossroad $ of crossroads.  \\

\begin{algorithmic}[1]
\ForAll{$l \in L$} 
\label{bigdeath} \State $x_d(l) \gets 1$                    \Comment{Death time temporarily set to maximum}
\EndFor                                     \Comment{End of initialisation}
\For{$\crossroad \in \oCrossroad $}                 \Comment{Consider potential crossroads timewise} \label{a1loop}
 \State $l_1, l_2 \gets l(\crossroad )$
 \If{$[x_b(l_1) < \crossroad < x_d(l_1)] \land  [x_b(l_2) < \crossroad < x_d(l_2)]$}  \label{a1cond1}
\\ \Comment{Do the lines cross?}
  \If{$M(l_1) = 0$}                          \Comment{Which line is killed?}
   \State $x_d(l_1) \gets \crossroad $      \Comment{Kill it}
   \State $M(l_2) \gets M(l_2) - 1$     \Comment{Count that $l_2$ killed it}
  \Else
   \State $x_d(l_2) \gets \crossroad $
   \State $M(l_1) \gets M(l_1) - 1$
  \EndIf 
 \EndIf 
\EndFor 

\Return P 
\end{algorithmic}

\end{algorithm}

\begin{figure}
\centering
\subfloat[][]{
             \label{rec1a}
             \includegraphics[width=.49\textwidth]{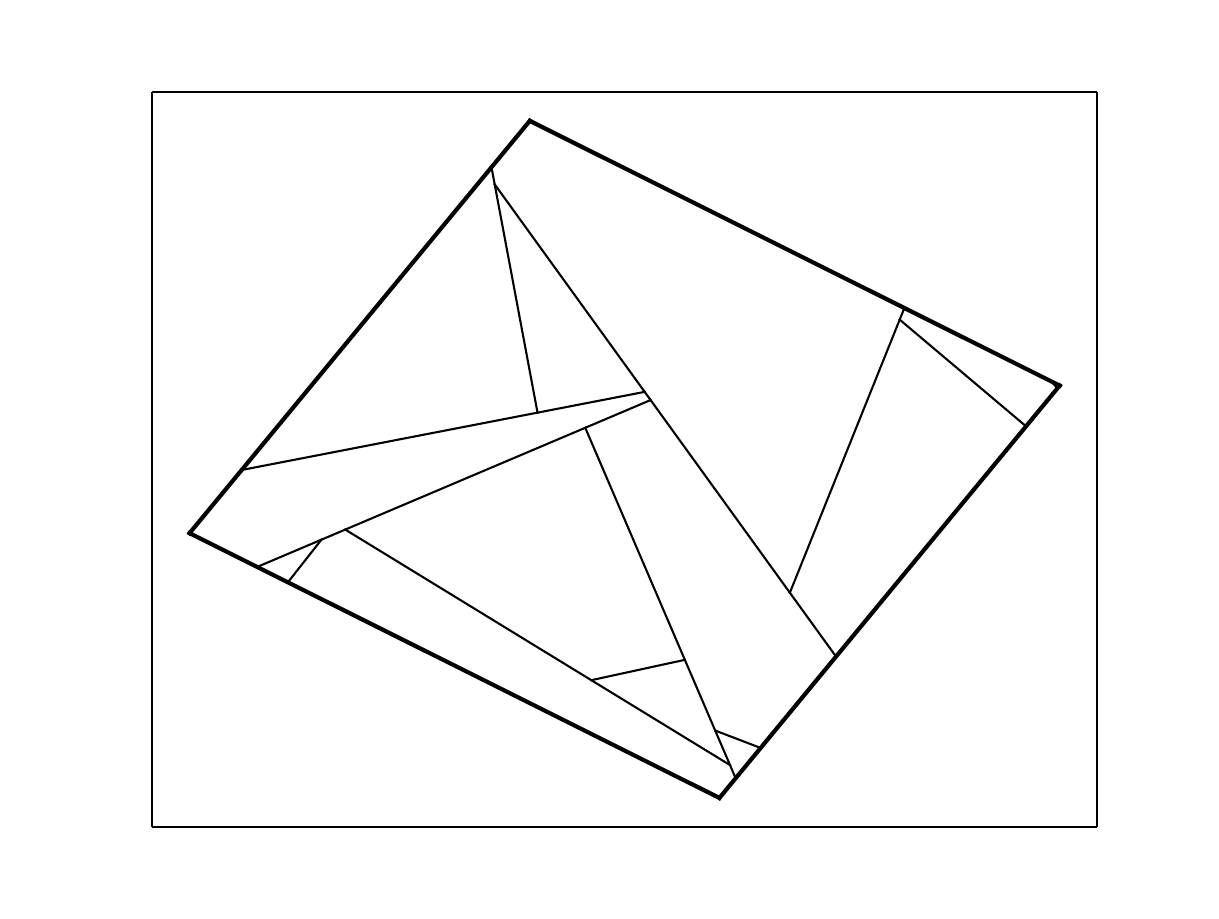}
             }
\subfloat[][]{
             \label{rec1b}
             \includegraphics[width=.49\textwidth]{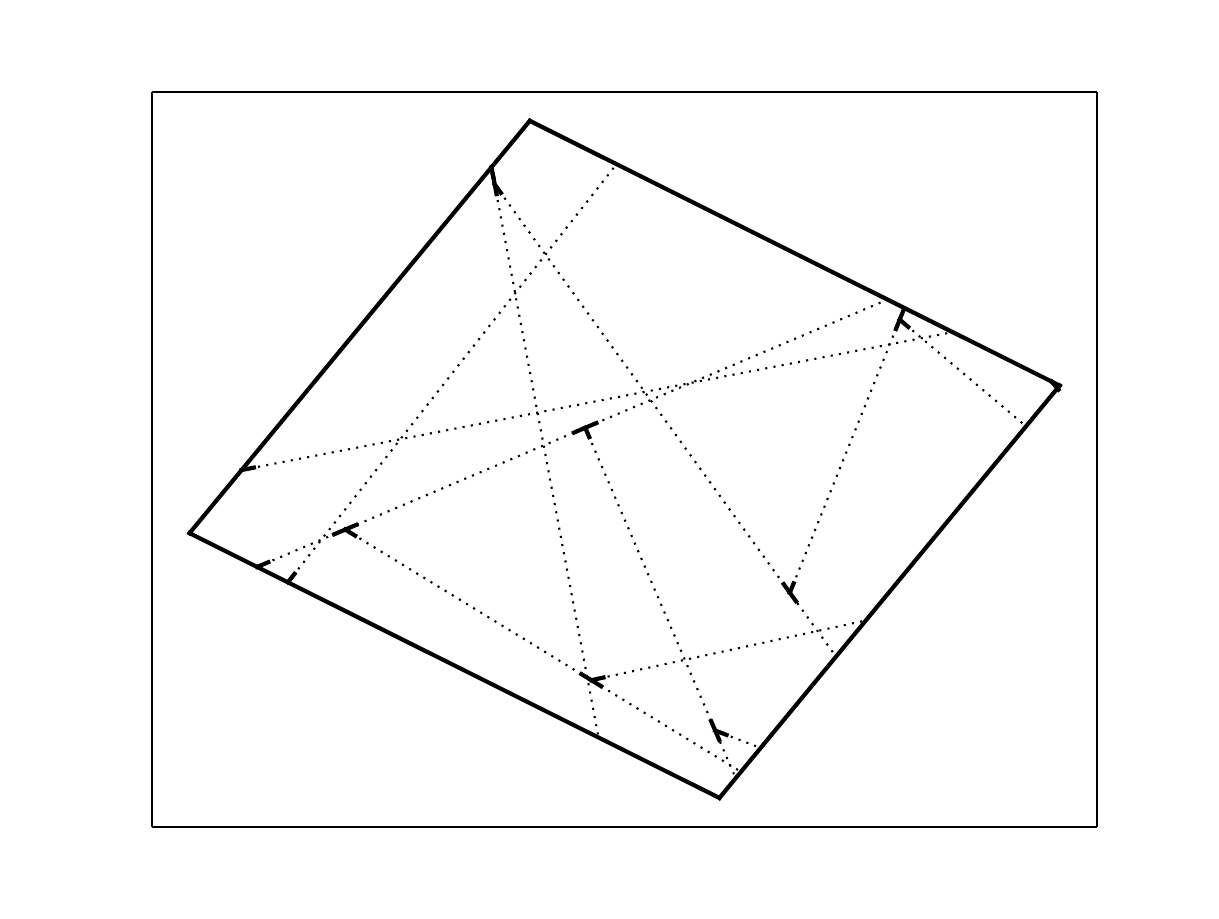}
             }\\
\subfloat[][]{
             \label{rec1c}
             \includegraphics[width=.49\textwidth]{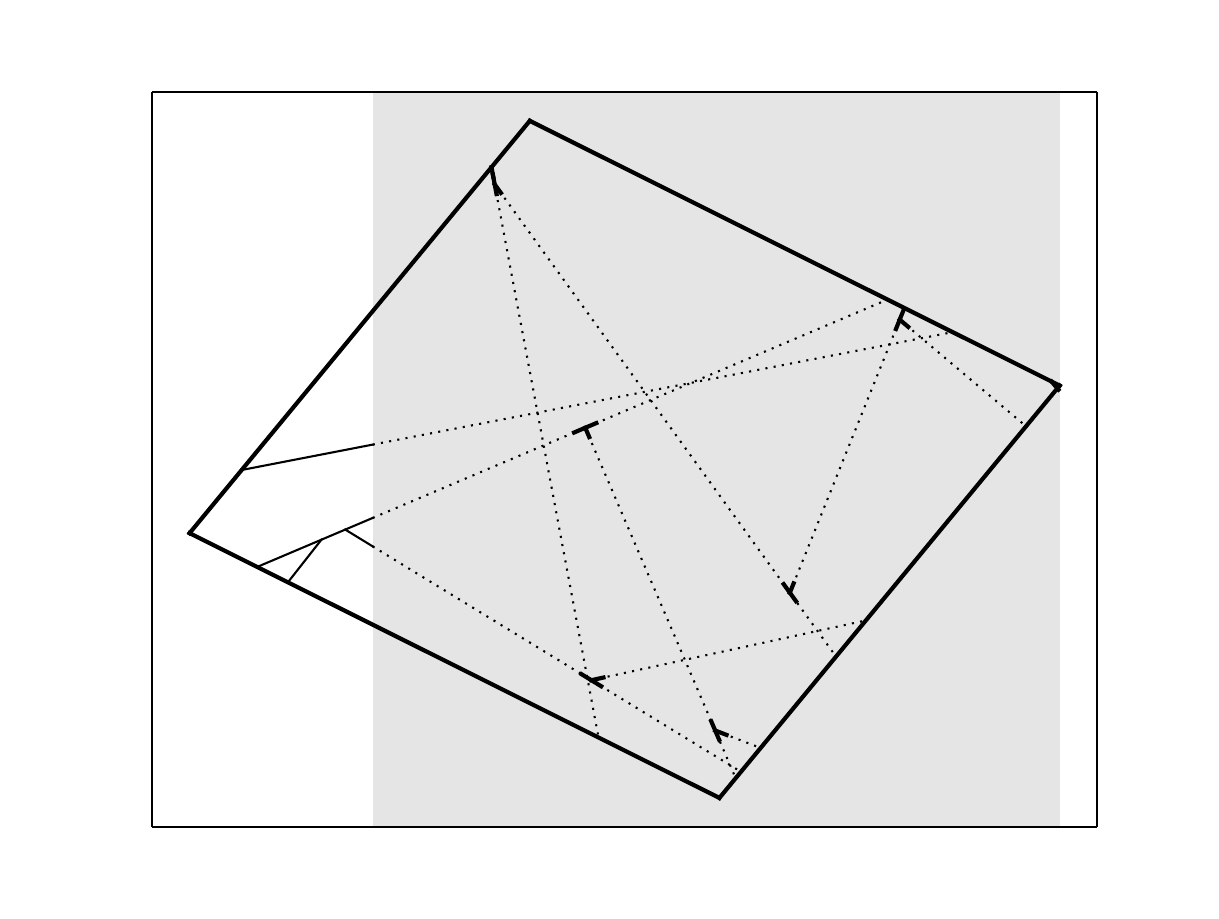}
             }
\subfloat[][]{
             \label{rec1d}
             \includegraphics[width=.49\textwidth]{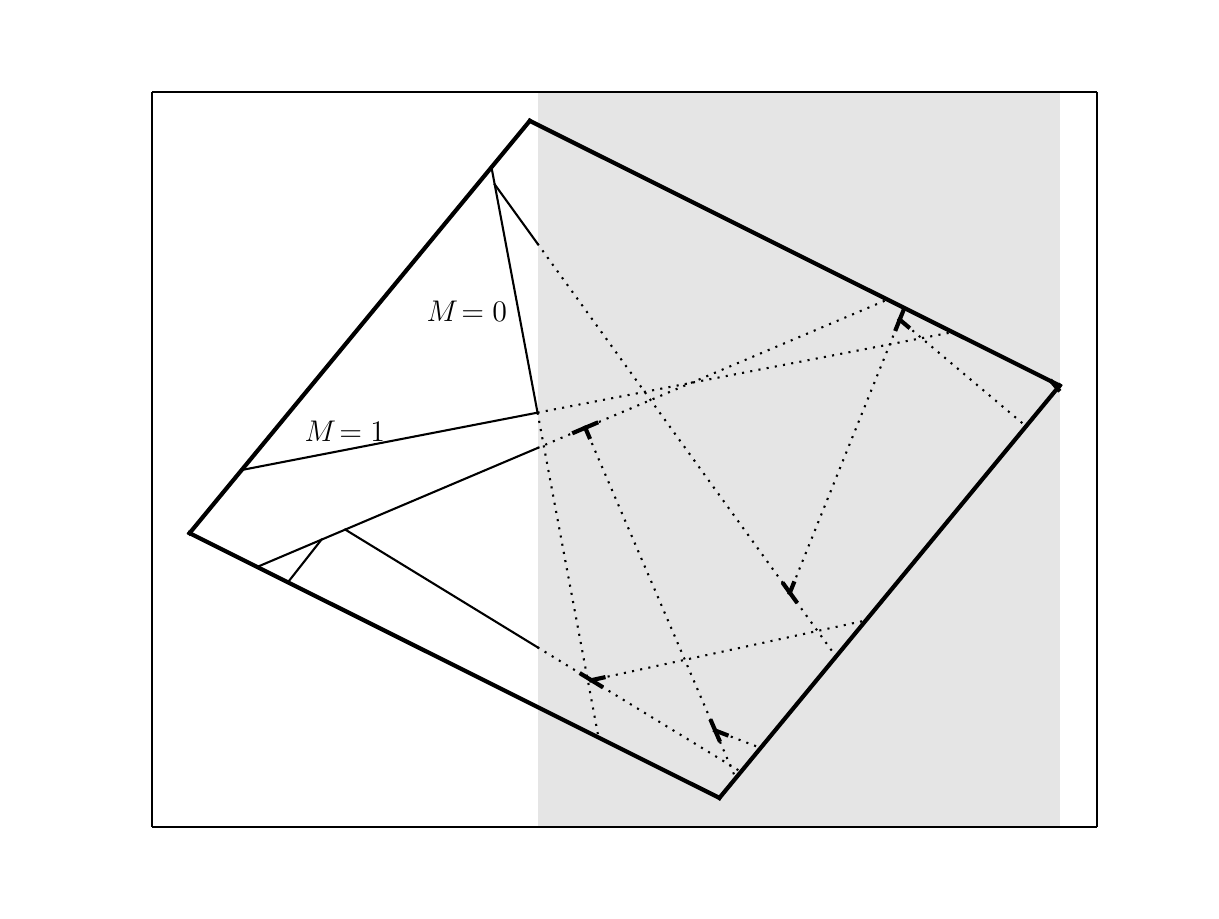}
             }\\
\subfloat[][]{
             \label{rec1e}
             \includegraphics[width=.49\textwidth]{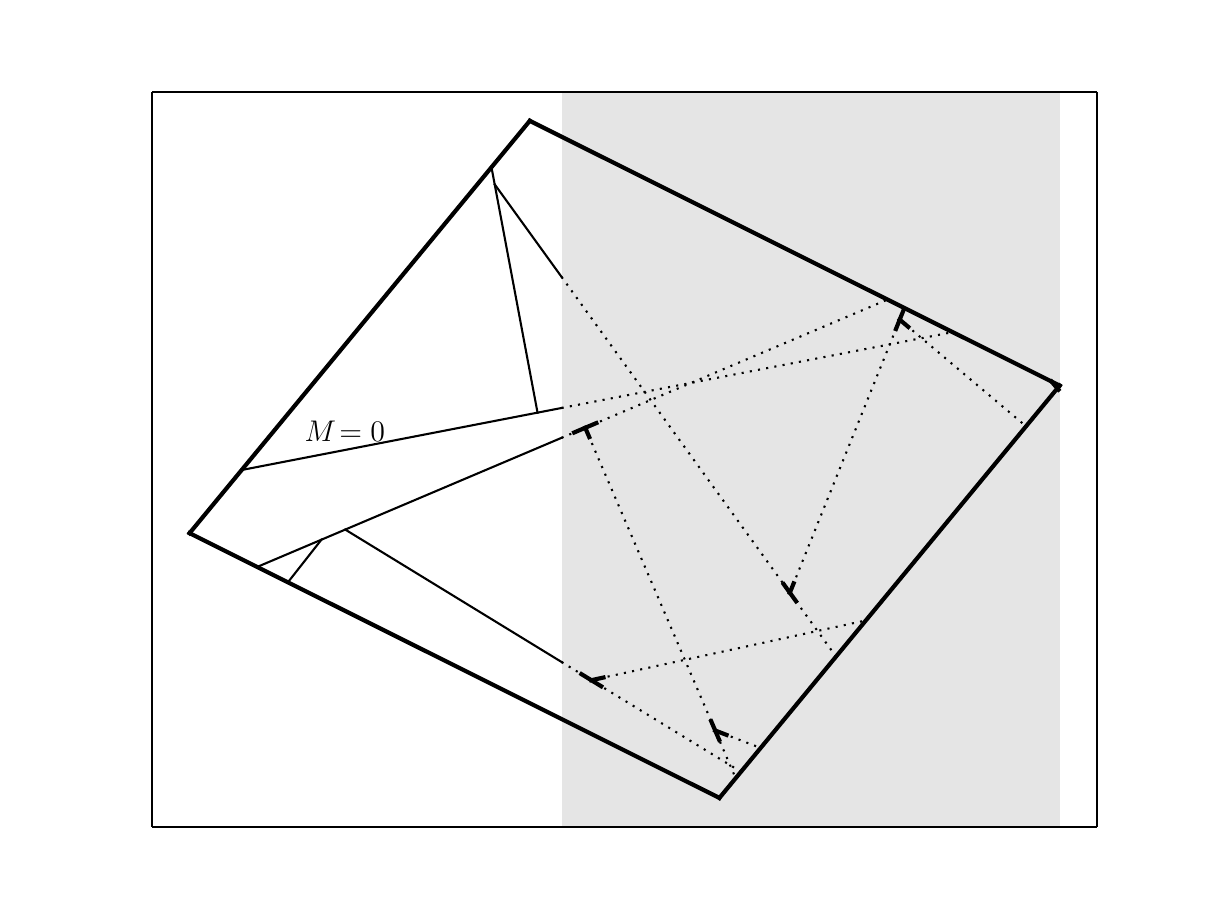}
             }
\subfloat[][]{
             \label{rec1f}
             \includegraphics[width=.49\textwidth]{rec1a.pdf}
             }
             \caption{\label{fig_algo1}\protect\subref{rec1a} is the tessellation to be rebuilt. We start knowing the lines, the birth times \protect\subref{rec1b} and the murder function. We move along the abcissas axis and prolongate the live segments \protect\subref{rec1c}. When two segments cross \protect\subref{rec1d}, we kill the one with zero murders left and decrease the other's counter \protect\subref{rec1e}. At the end of the pass, we get the tessellation \protect\subref{rec1f}.}
\end{figure}

\begin{algorithm}
\caption{Rebuild from final encoding}
\label{main}
\ \\
\noindent
\textbf{Input:} The lines $L$, the ordered and reverse-ordered list of crossroads $\oCrossroad$ and $\aCrossroad$, a subset $U_0 \subset L$ of orphan lines satisfying requirement~\eqref{binary}, a prototessellation $P = (x_b, x_d)$ such that $x_b(l) = x_b^T(l)$ for all non-orphan line $l \in L\backslash U_0$, a function ``virtual murders'' $V: L \cup \boundary \to \mathbb{N} $ defined in \eqref{Vl} , and a function $O^T: L \cup \boundary \in \mathbb{N} $ giving the number of orphan children a line has \eqref{Others}.\ \\
\begin{algorithmic}[1]
\State $U \gets U_0$
\ForAll{$l \in L$} 
 \State $x_d(l) \gets 1$                    \Comment{Death time set to maximum for now} \label{death1}
\EndFor  
\ForAll{$l \in U$}
 \State $x_b(l) \gets 1$                    \Comment{Birth of orphans set to maximum for now} \label{birth1} 
\EndFor 
                                            \Comment{End of ``preinitialisation''}   \label{end_preinitialisation}
\ForAll{$\crossroad \in \oCrossroad$}                                                \label{forloop}
 \State $l_1, l_2 \leftarrow l(\crossroad )$
 \If{$[x_b(l_1) = \crossroad] \land [x_b(l_2) > \crossroad]$} \Comment{Is $l_1$ the first child of $l_2$?}
  \State $x_b(l_2) \gets \crossroad$                  \Comment{Temporary maximum birth time for $l_2$} \label{birth2} 
 \ElsIf{$[x_b(l_2) = \crossroad] \land [x_b(l_1) > \crossroad]$} 
  \State $x_b(l_1) \gets \crossroad$                                                                   \label{birth3}
  \ElsIf{$[x_b(l_1) < \crossroad < x_d(l_1)] \land  [x_b(l_2) < \crossroad < x_d(l_2)]$} \label{vmurder} \\ \Comment{Do the lines cross?}         
  \If{$V(l_1) = 0$}                               \Comment{Is $l_1$ ``virtual-killed''?}
   \State $x_d(l_1) \gets \crossroad$                 \Comment{Death time to new maximum} \label{death2}
   \If{$V(l_2) = 0$}                              \Comment{Is $l_2$ ``virtual-killed''?}
    \State $x_d(l_2) \gets \crossroad$                \Comment{Death time to new maximum} \label{death3}
   \Else
   \State $V(l_2) \gets V(l_2) - 1$              \Comment{Count that $l_2$ ``virtual-killed'' $l_1$}    \label{vkill}
   \EndIf 
  \Else                                           \Comment{In that case $l_2$ is ``virtual-killed''}
   \State $V(l_1) \gets V(l_1) - 1$               \Comment{Count that $l_1$ ``virtual-killed'' $l_2$}
   \State $x_d(l_2) \gets \crossroad$                 \Comment{Death time to new maximum}               \label{death4}
  \EndIf
 \EndIf 
\EndFor                                           \Comment{End of initialisation}

\Repeat
 \State $P \gets Parent\_seek(L, P, \oCrossroad, U)$               \Comment{Extend backwards}
 \State $\Cuts, U, P \gets Cutting(L, P, \aCrossroad, U_0, O^T)$    \Comment{Cut too long segments}
\Until{$\Cuts = 0$} 

\Return $P$
\end{algorithmic}
\end{algorithm}

\begin{algorithm}
\caption{Parent-seeking loop}
\label{parent}
\ \\
\noindent
\textbf{Input:} The lines $L$, a pretessellation $P = (x_b, x_d)$, the ordered sequence of crossroads $\oCrossroad$, a subset $U \subset L$ of lines whose parent is not currently known. \ \\
\begin{algorithmic}[1]

\ForAll{$\crossroad \in \aCrossroad$}                     \Comment{Reverse timewise}
 \State $l_1, l_2 \leftarrow l(\crossroad )$         
 \If{$[l_1 \in U] \land [x_b(l_1) > \crossroad] \land [x_b(l_2) < \crossroad < x_d(l_2)]$}                 \label{cond1} 
  \State $x_b(l_1) \gets \crossroad$                   \Comment{Extend $l_1$ backwards}                \label{Parent_birth1} 
  \State $U \gets U - l_1$                         \Comment{$l_1$ seen as born on $l_2$, for now}  \label{removeU1}
 \EndIf 
 \If{$[l_2 \in U] \land [x_b(l_2) > \crossroad] \land [x_b(l_1) < \crossroad < x_d(l_1)]$}                \label{cond2} 
  \State $x_b(l_2) \gets \crossroad$                   \Comment{Same as above, $l_1$ and $l_2$ switched} \label{Parent_birth2}
  \State $U \gets U - l_2$                                                                           \label{removeU2}
 \EndIf
\EndFor

\Return $P$

\end{algorithmic}
\end{algorithm}

\begin{algorithm}
\caption{Cutting loop}
\label{cuts}
\ \\
\noindent
\textbf{Input:} The lines $L$, a pretessellation $P = (x_b, x_d )$, the reverse-ordered sequence of crossroads $\aCrossroad$, a subset $U_0 \subset L$ of orphan lines, a function $O^T: L \cup \boundary \to \mathbb{N} $ giving the number of orphan children a line has, and a variable set $U \subset L$ initially empty. \ \\
\begin{algorithmic}[1]

\State $\Cuts \gets 0$                             \Comment{Reset number of cuts}
\ForAll{$l \in L$}
 \State $O(l) \gets 0$                             \Comment{Reset number of orphan children}
\EndFor 

\ForAll{$\crossroad \in \oCrossroad$}                      \Comment{Timewise}
 \State $l_1, l_2 \leftarrow l(\crossroad )$     
 \If{$[l_1 \in U_0] \land [x_b(l_1) = \crossroad]$}                      \label{current_birth}
  \State $O(l_2) \gets O(l_2) + 1$                 
  \If{$O(l_2) = O^T(l_2) + 1$}                                              \label{trop}
   \State $x_d(l_2) \gets \crossroad$                  \Comment{Death time to new maximum}        \label{Cdeath1}
   \State $\Cuts \gets 1$                                                                     \label{cuts1}
  \EndIf 
  \If{$O(l_2) > O^T(l_2)$}                        \Comment{$l_2$ has too many orphan children}
   \State $l_1 \in U$                              \Comment{We do not know the father of $l_1$} \label{U1}
  \EndIf                                                         \label{finsym1}
 \ElsIf{$[l_2 \in U_0] \land [x_b(l_2) = \crossroad]$} \Comment{Same, switching $l_1$ and $l_2$} \label{current_birth2}
  \State $O(l_1) \gets O(l_1) + 1$                                                         
  \If{$O(l_1) = O^T(l_1) + 1$}                                                                       
   \State $x_d(l_1) \gets \crossroad$                                                       \label{Cdeath2}         
   \State $\Cuts \gets 1$                                                                      \label{cuts2}
  \EndIf 
  \If{$O(l_1) > O^T(l_1)$}                       
   \State $l_2 \in U$                                                           \label{U2}                       
   \EndIf                                                                        \label{finsym2}
 \EndIf 
\EndFor

\Return $\Cuts, U, P$

\end{algorithmic}
\end{algorithm}

\begin{figure}
\centering
\subfloat[][]{
             \label{rec20}
             \includegraphics[width=.49\textwidth]{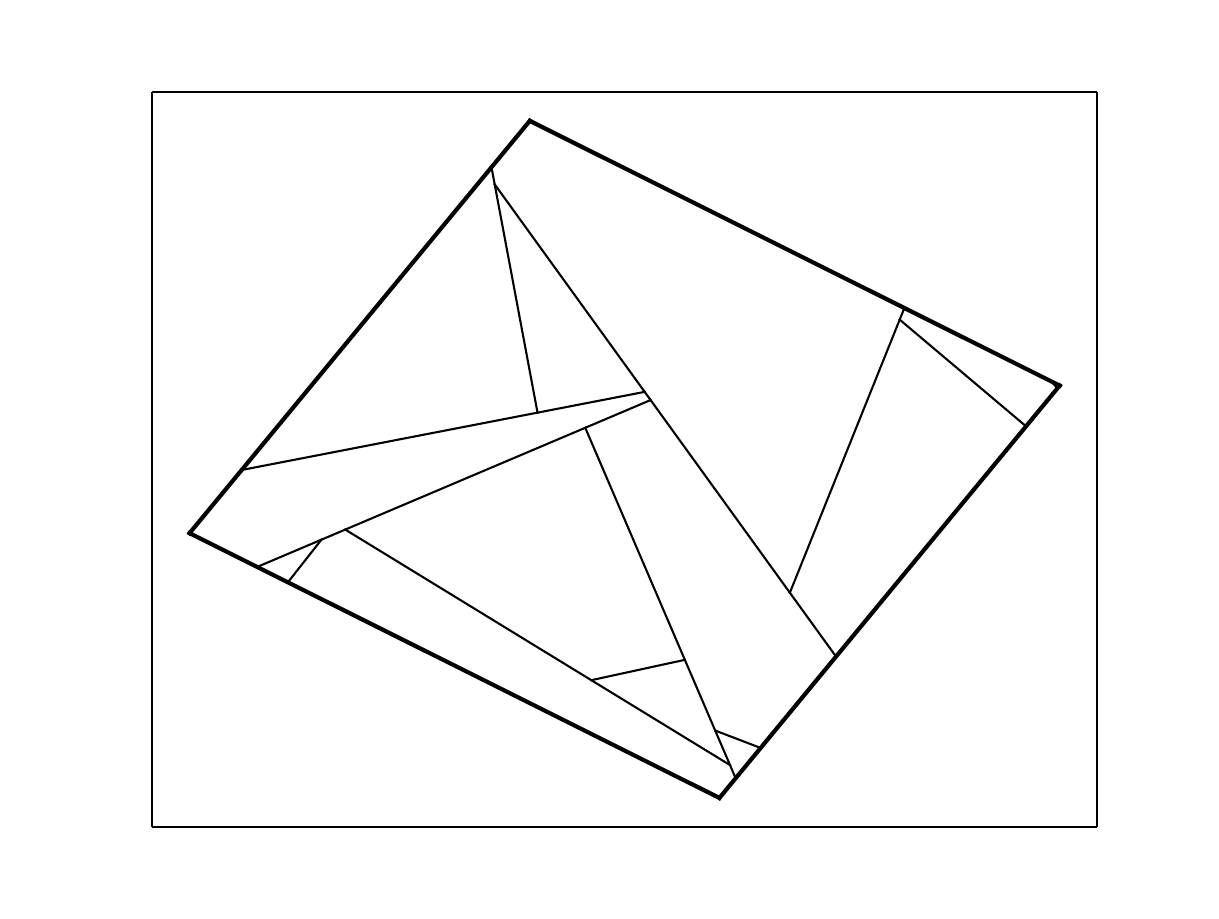}
         }
\subfloat[][]{
             \label{rec21}
             \includegraphics[width=.49\textwidth]{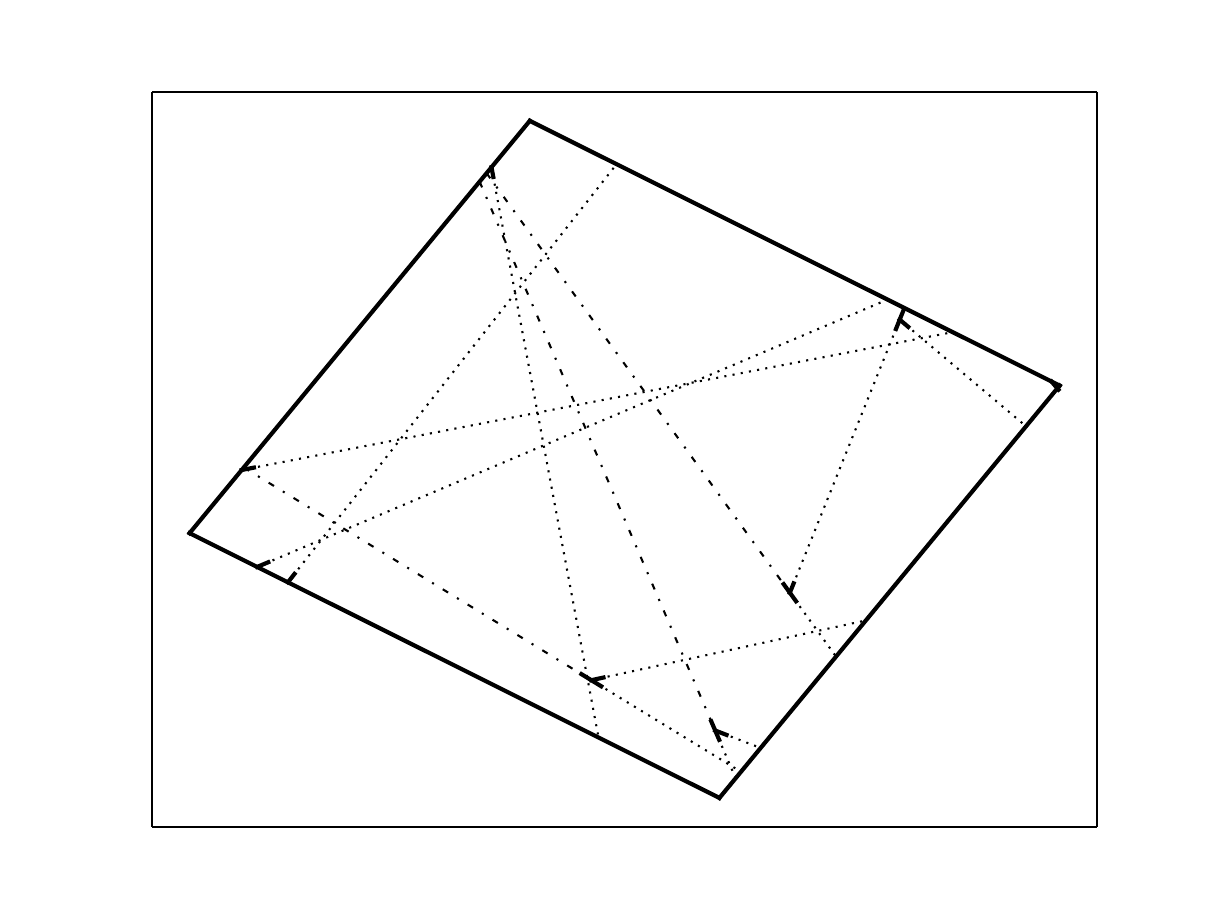}
             }\\
\subfloat[][]{
             \label{rec2a}
             \includegraphics[width=.49\textwidth]{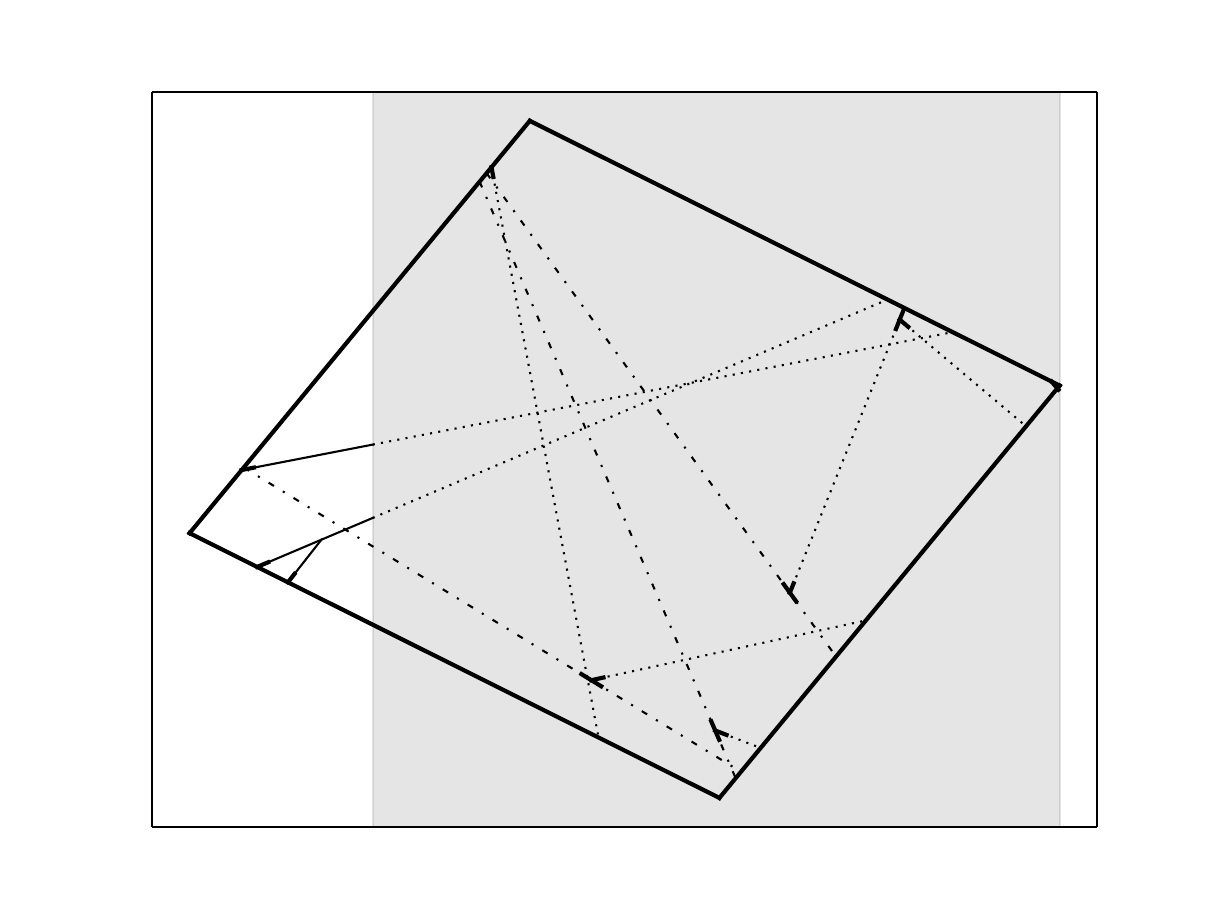}
             }
\subfloat[][]{
             \label{rec2b}
             \includegraphics[width=.49\textwidth]{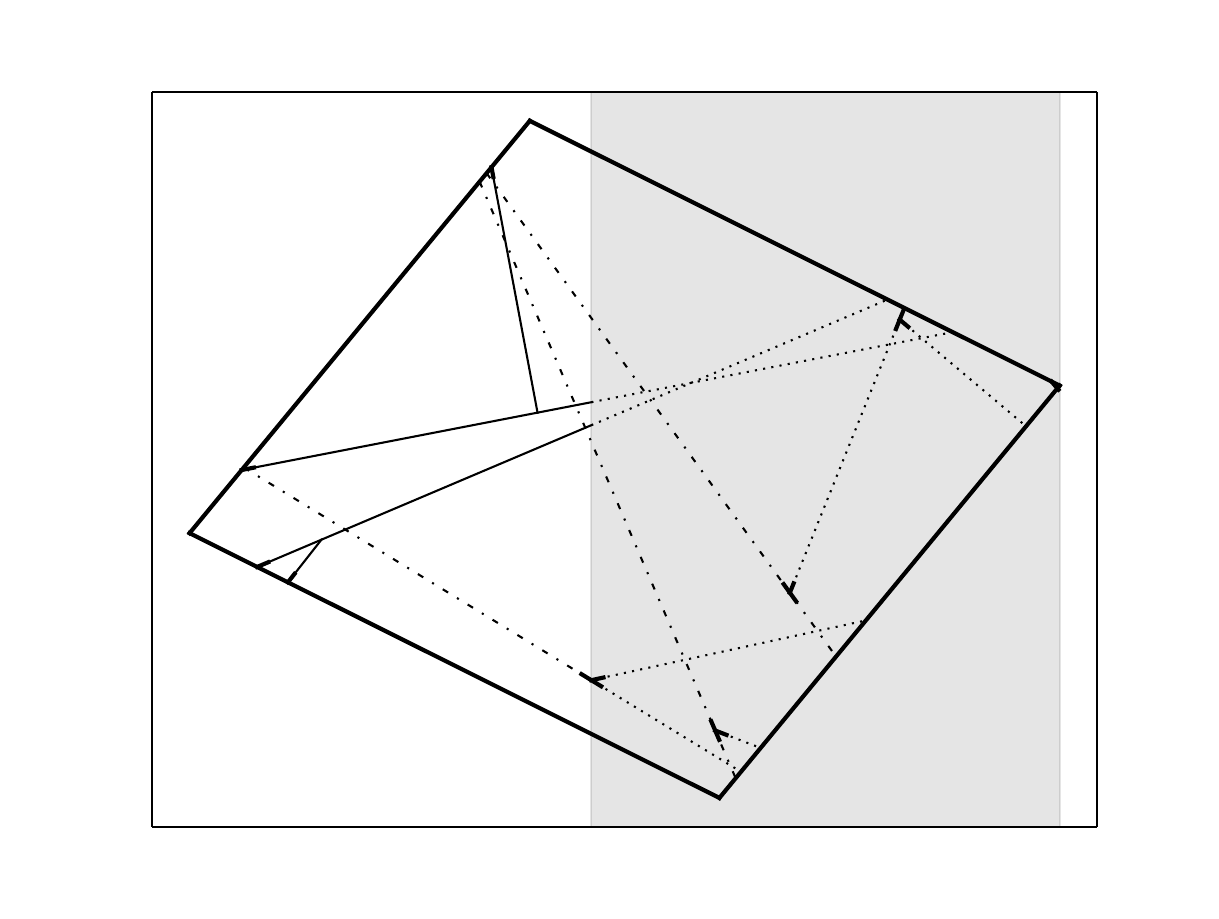}
             }\\
\subfloat[][]{
             \label{rec2c}
             \includegraphics[width=.49\textwidth]{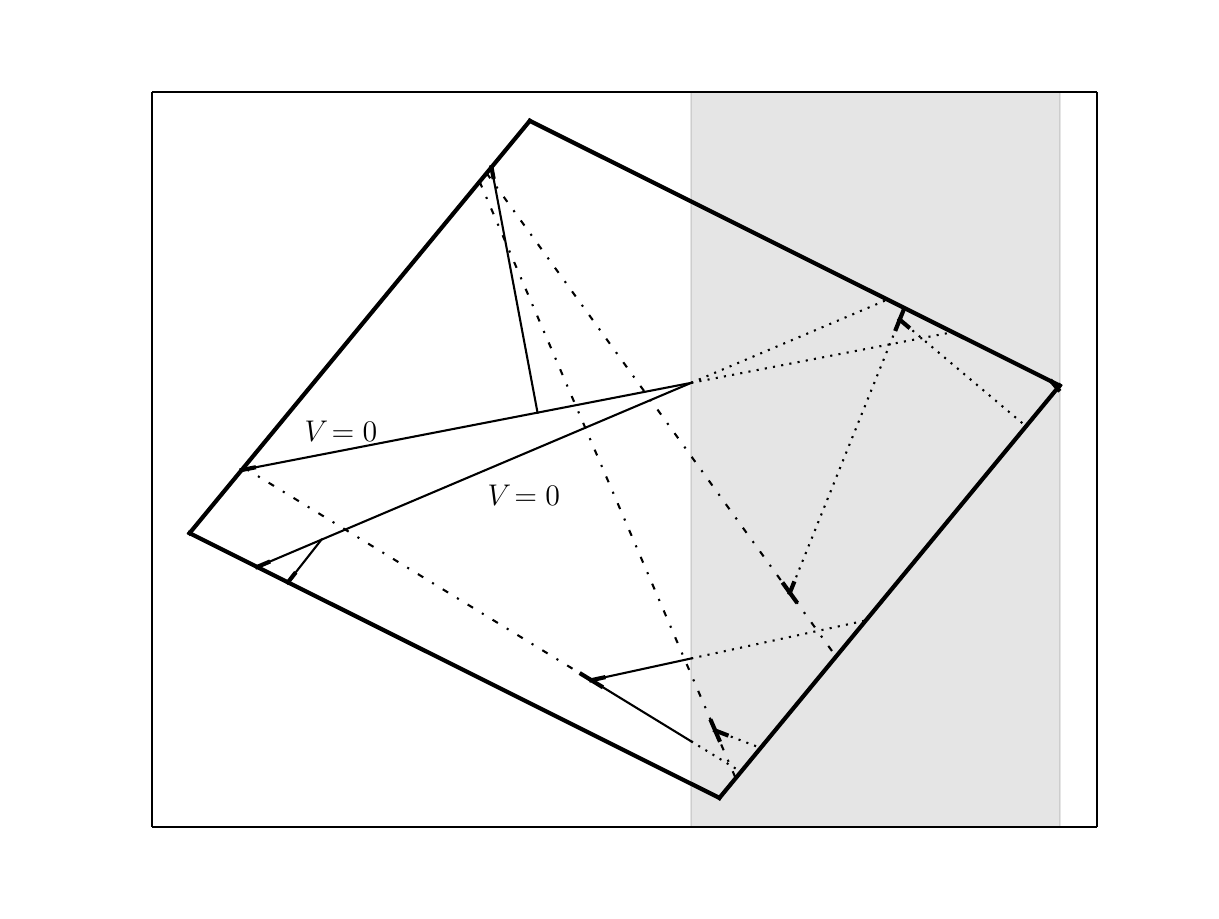}
             }
\subfloat[][]{
             \label{rec2d}
             \includegraphics[width=.49\textwidth]{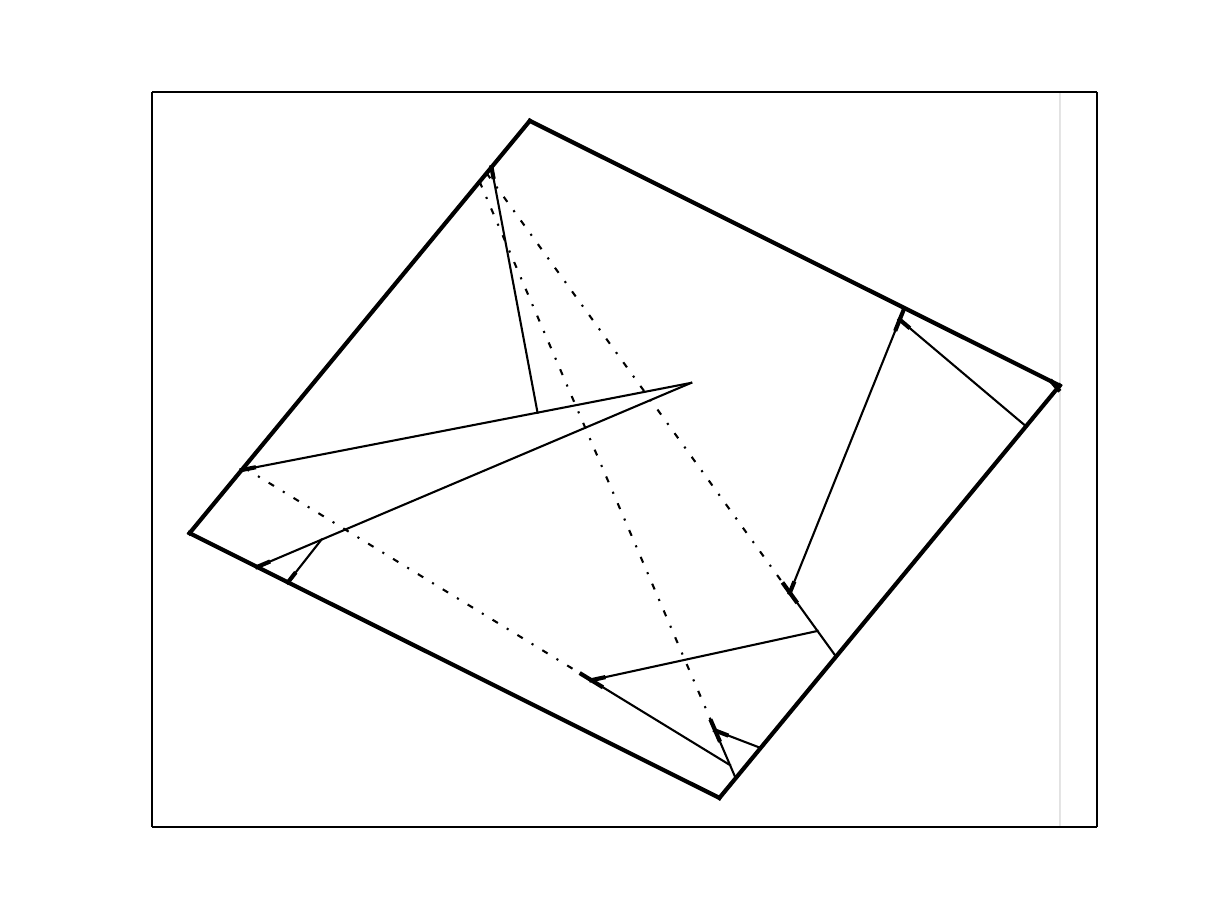}
             }\\
             \caption{\label{rec2}\small\protect\subref{rec20} is the tessellation to be rebuilt. We start knowing the lines, the birth times of some of the lines, the first child of the others \protect\subref{rec21} and the virtual murder and orphan children functions. The half-dotted lines correspond to the lines whose birth is unknown. We move along the abcissas axis and prolongate the live segments \protect\subref{rec2a}. When we reach the birth time of the first child of a line, we add both the line and its child to the live segments \protect\subref{rec2b}. When two segments cross, we kill those with zero virtual murders left \protect\subref{rec2c}. At the end of the initialisation pass, we get a pretessellation \protect\subref{rec2d}.}
\end{figure}

\begin{figure}
    \ContinuedFloat
    \centering
\subfloat[][]{
             \label{rec2e}
             \includegraphics[width=.49\textwidth]{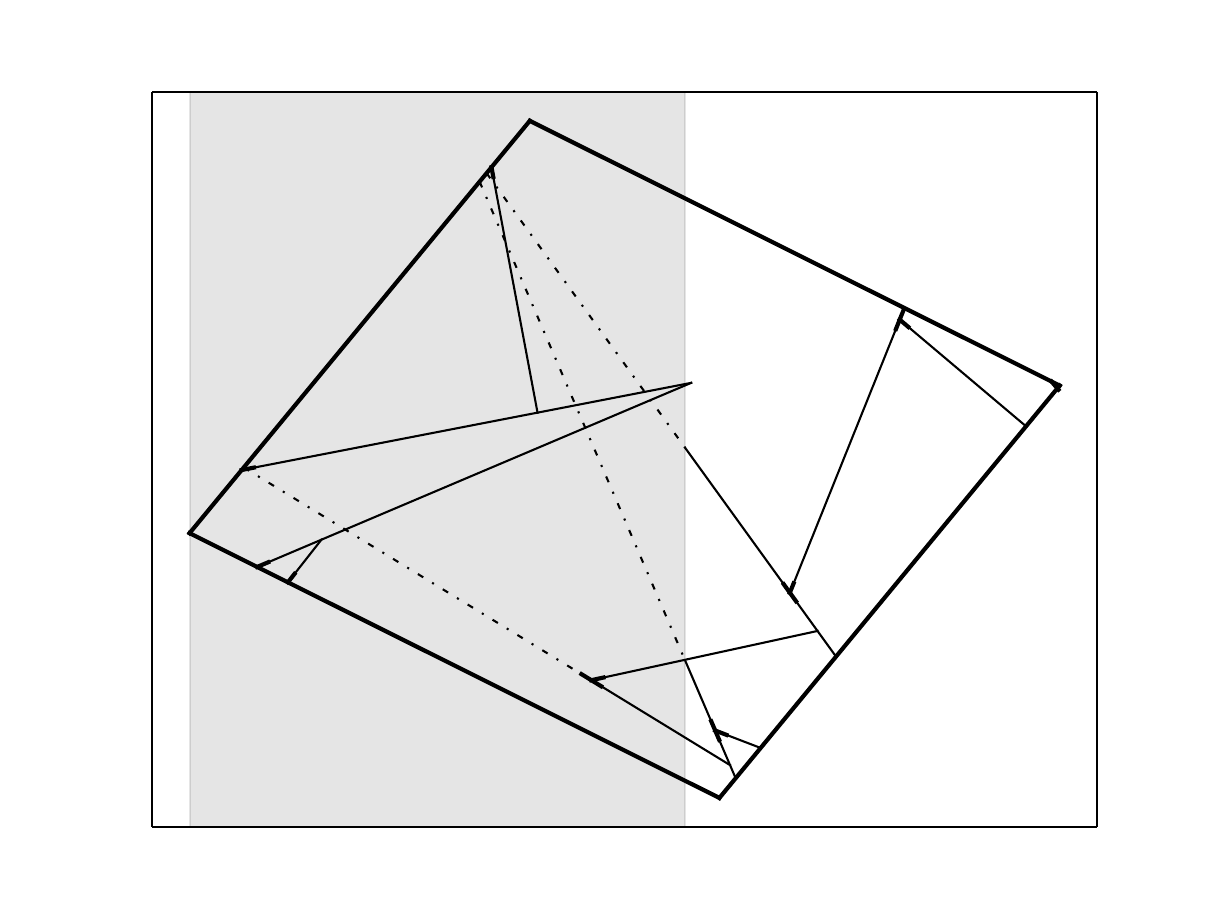}
             }
\subfloat[][]{
             \label{rec2f}
             \includegraphics[width=.49\textwidth]{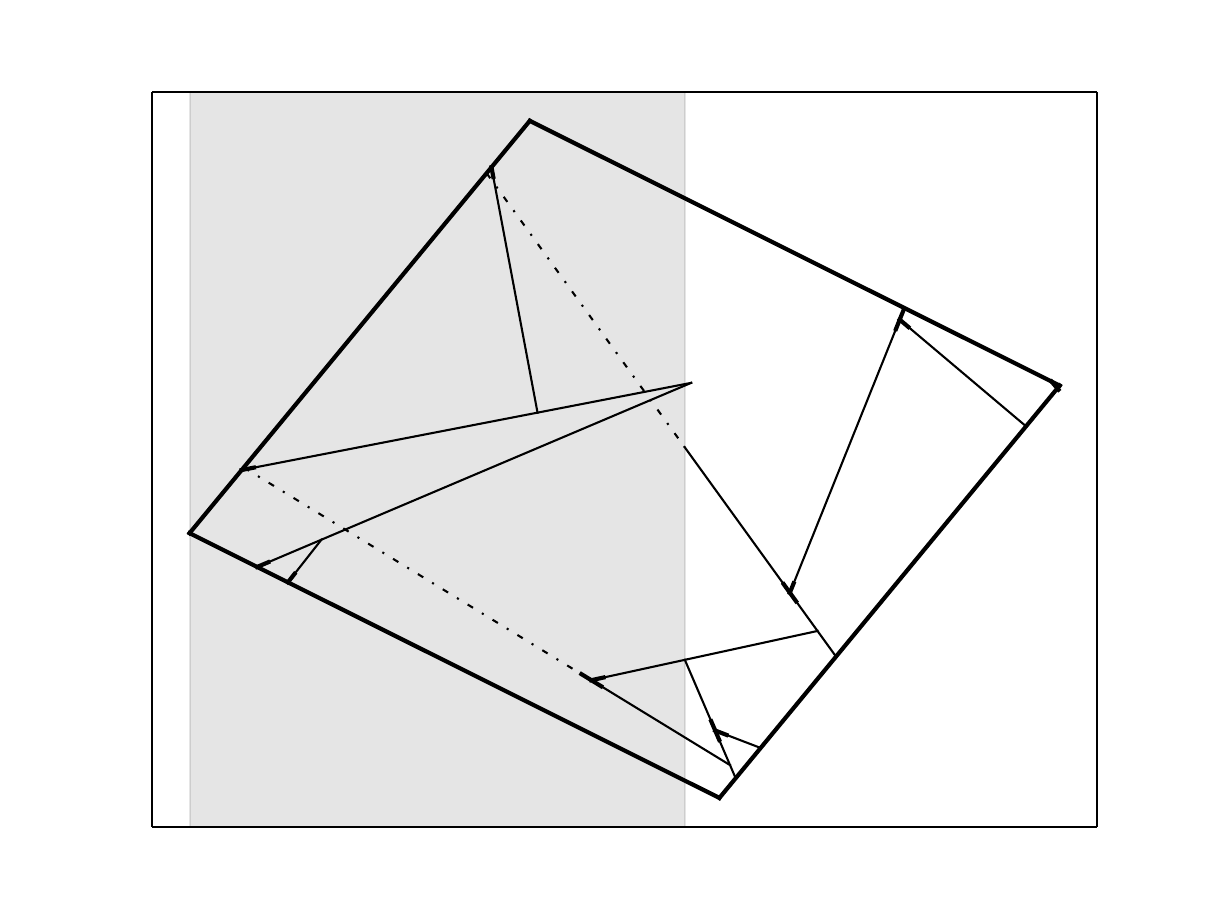}
             }\\
\subfloat[][]{
             \label{rec2g}
             \includegraphics[width=.49\textwidth]{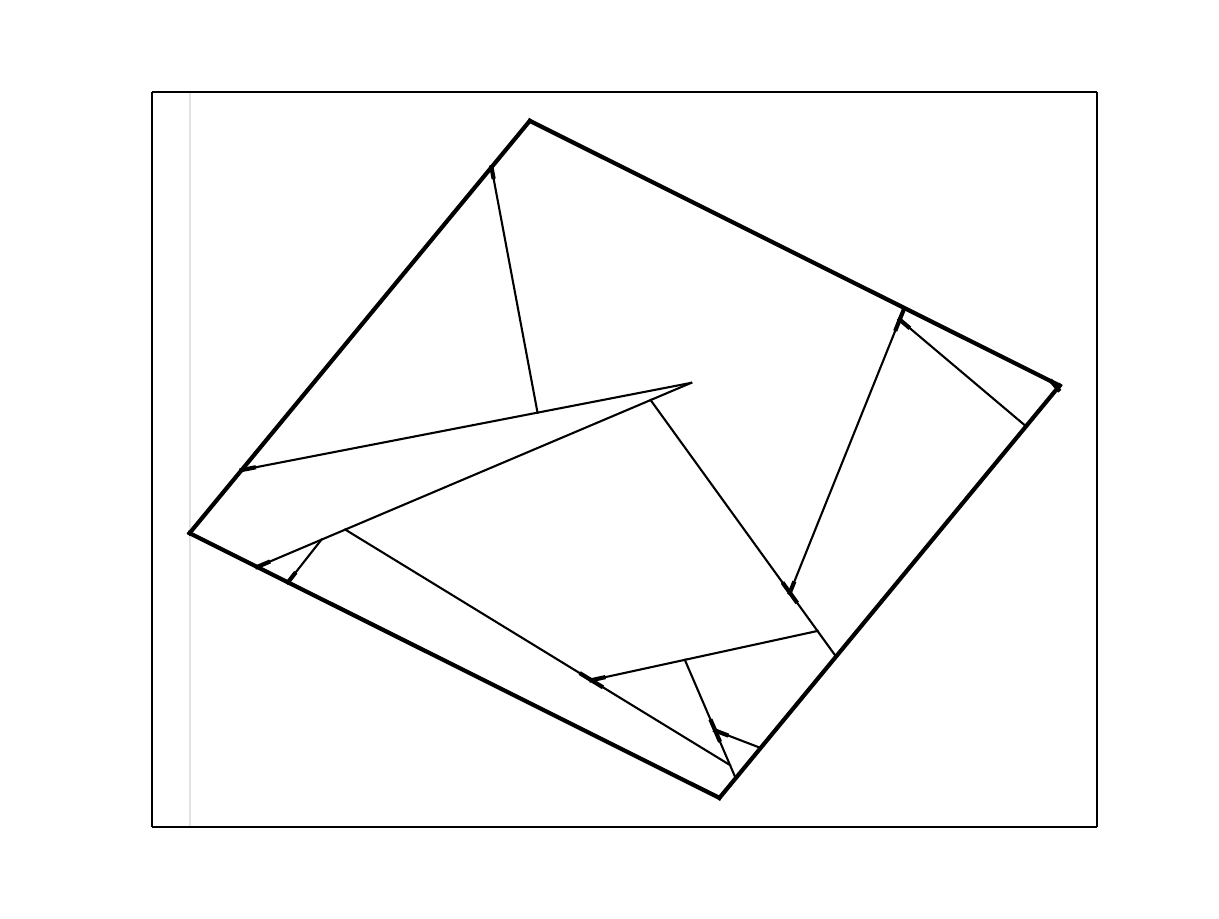}
             }
\subfloat[][]{
             \label{rec2h}
             \includegraphics[width=.49\textwidth]{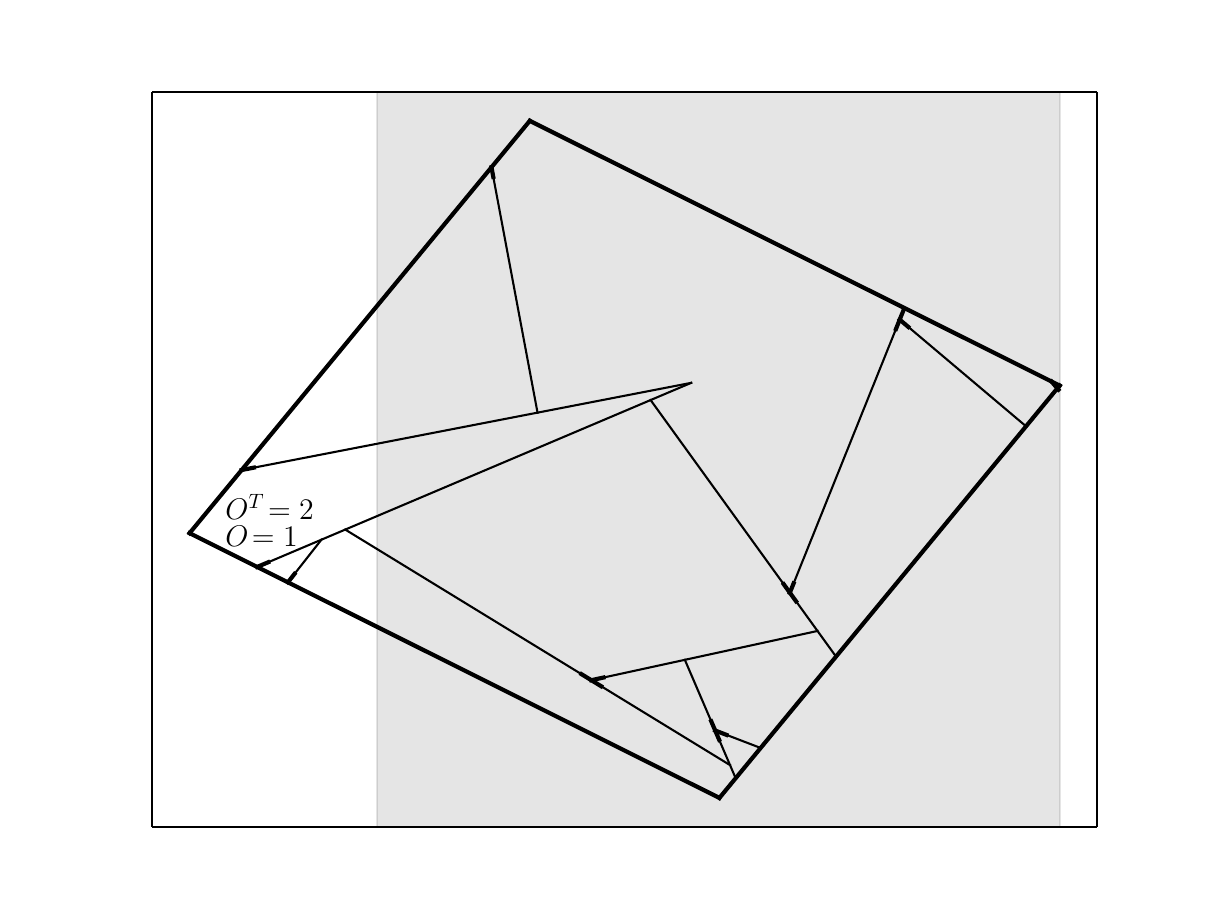}
             }\\
\subfloat[][]{
             \label{rec2i}
             \includegraphics[width=.49\textwidth]{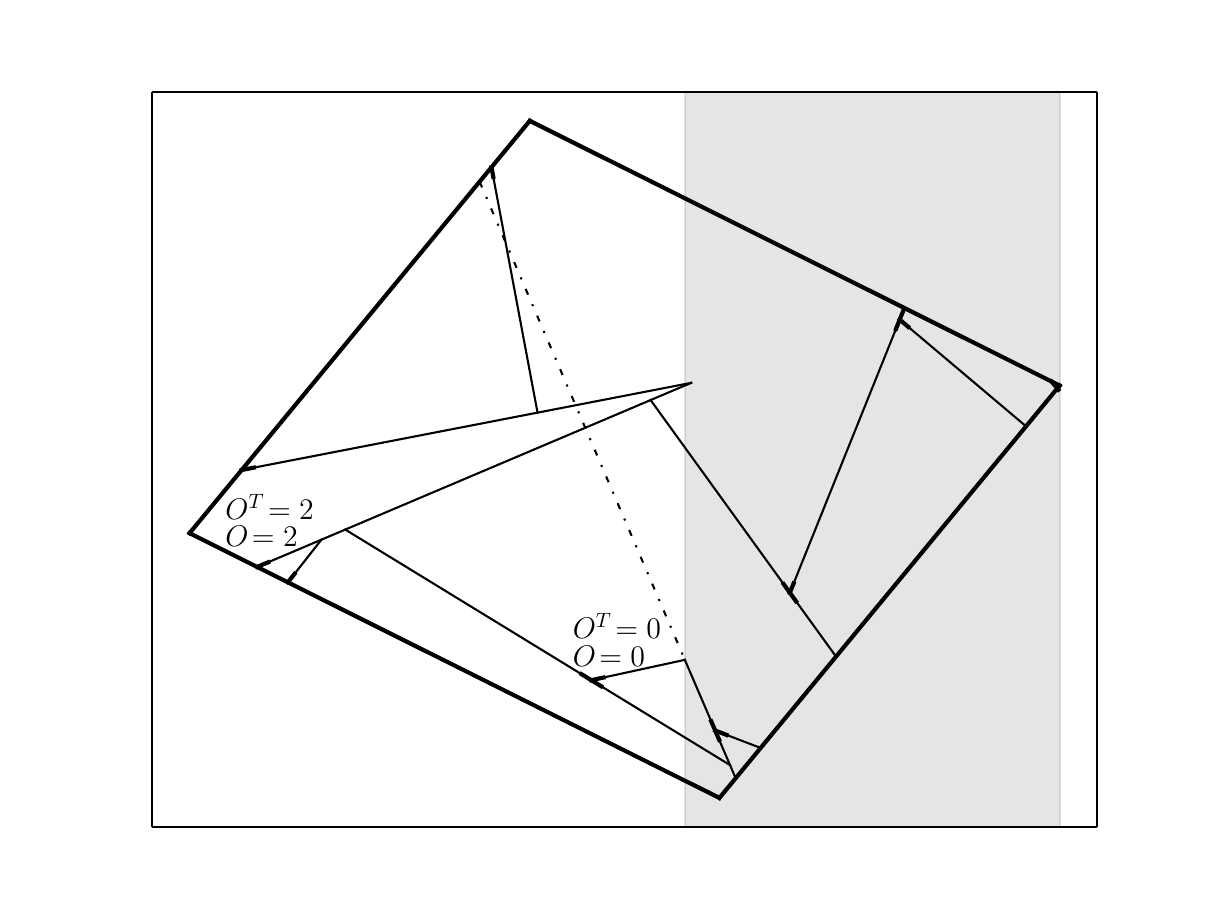}
             }
\subfloat[][]{
             \label{rec2j}
             \includegraphics[width=.49\textwidth]{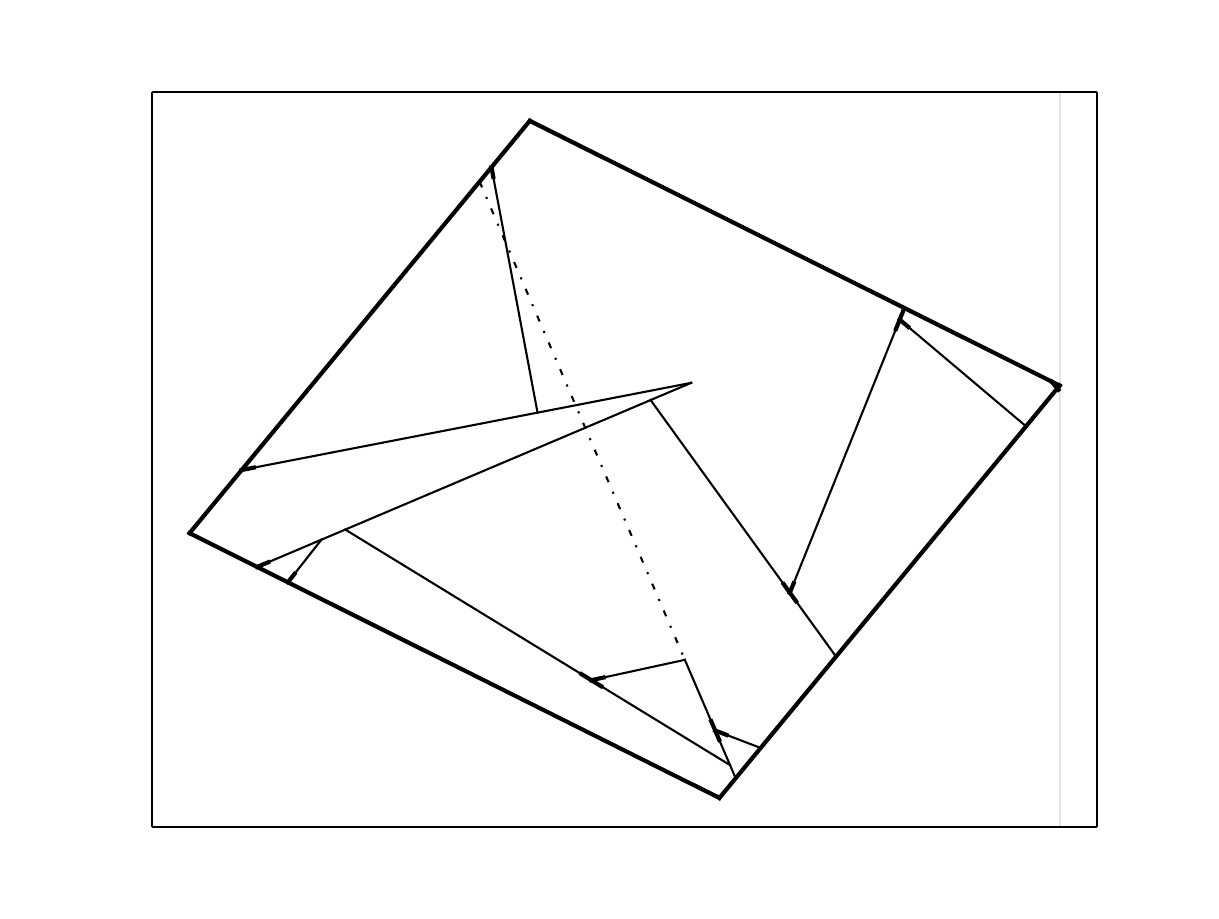}
             }
             \caption{We move backwards in time, prolongating the segments whose parent we do not know \protect\subref{rec2e} until they hit another segment \protect\subref{rec2f}. At the end of the pass, every segment has a putative parent \protect\subref{rec2g}. We move again forward, comparing the number of orphan children with that in the true tessellation \protect\subref{rec2h}. When it would be exceeded, the putative parent is cut \protect\subref{rec2i}. After the pass, no segment has too many children, but some parents are still unknown \protect\subref{rec2j}.}
\end{figure}

\begin{figure}
\ContinuedFloat
\centering
\subfloat[][]{
             \label{rec2k}
             \includegraphics[width=.49\textwidth]{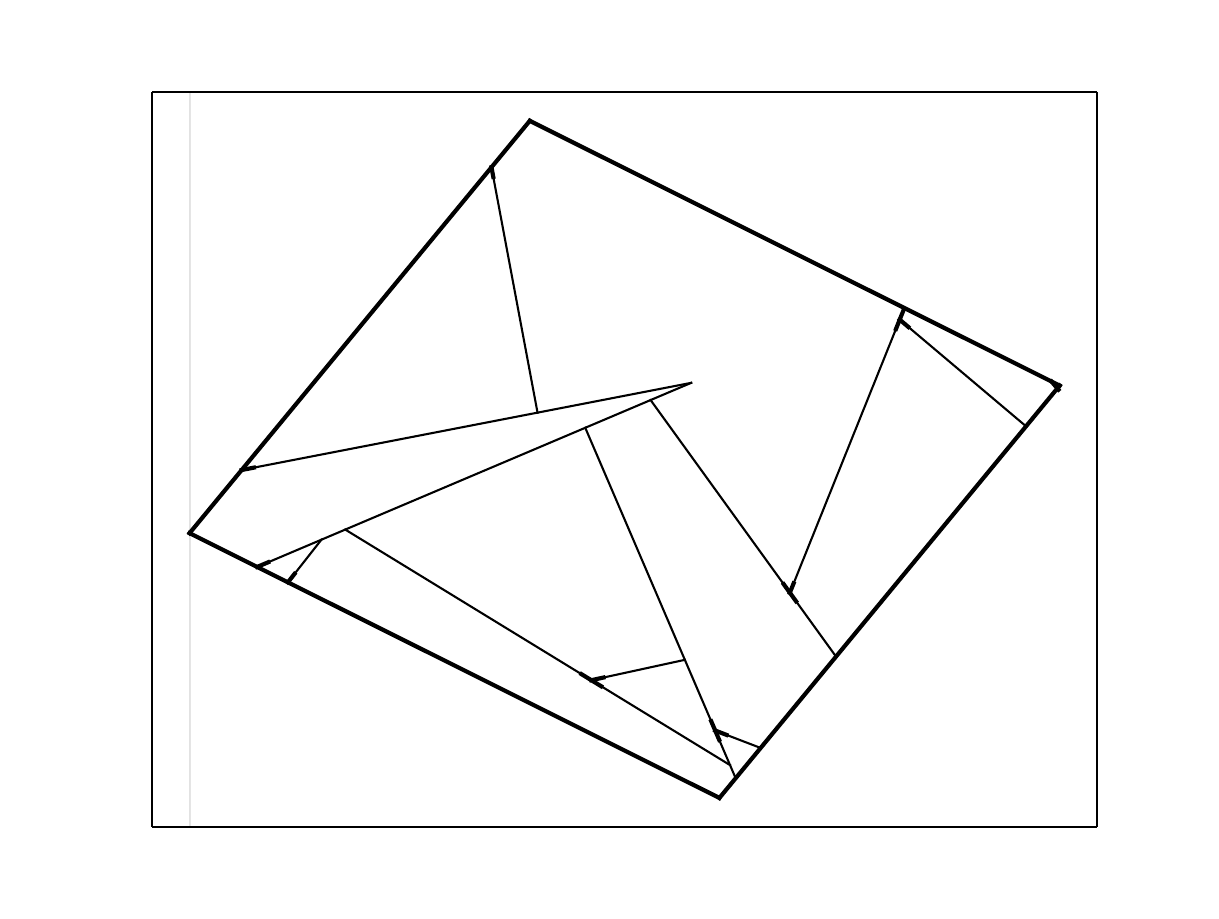}
             }
\subfloat[][]{
             \label{rec2l}
             \includegraphics[width=.49\textwidth]{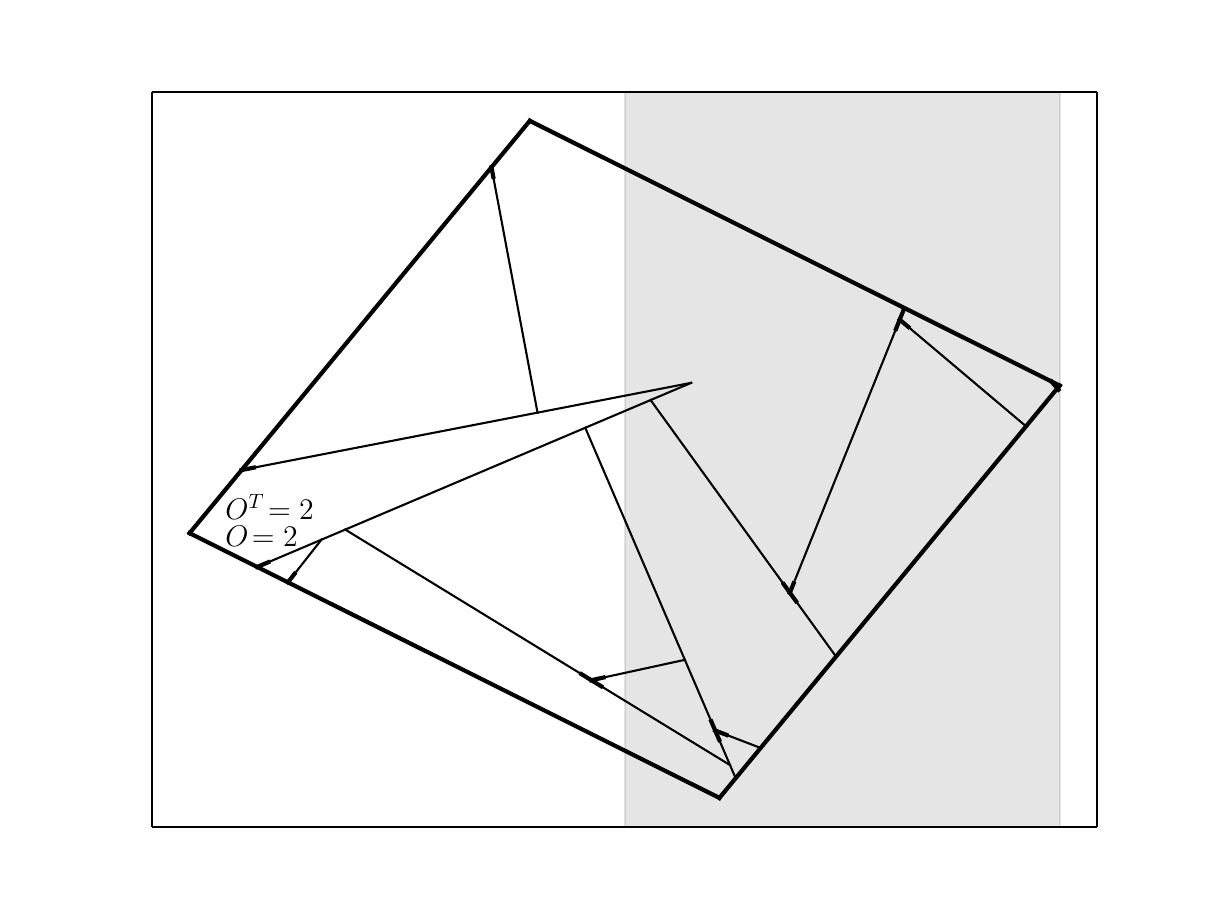}
             }\\
\subfloat[][]{
             \label{rec2m}
             \includegraphics[width=.49\textwidth]{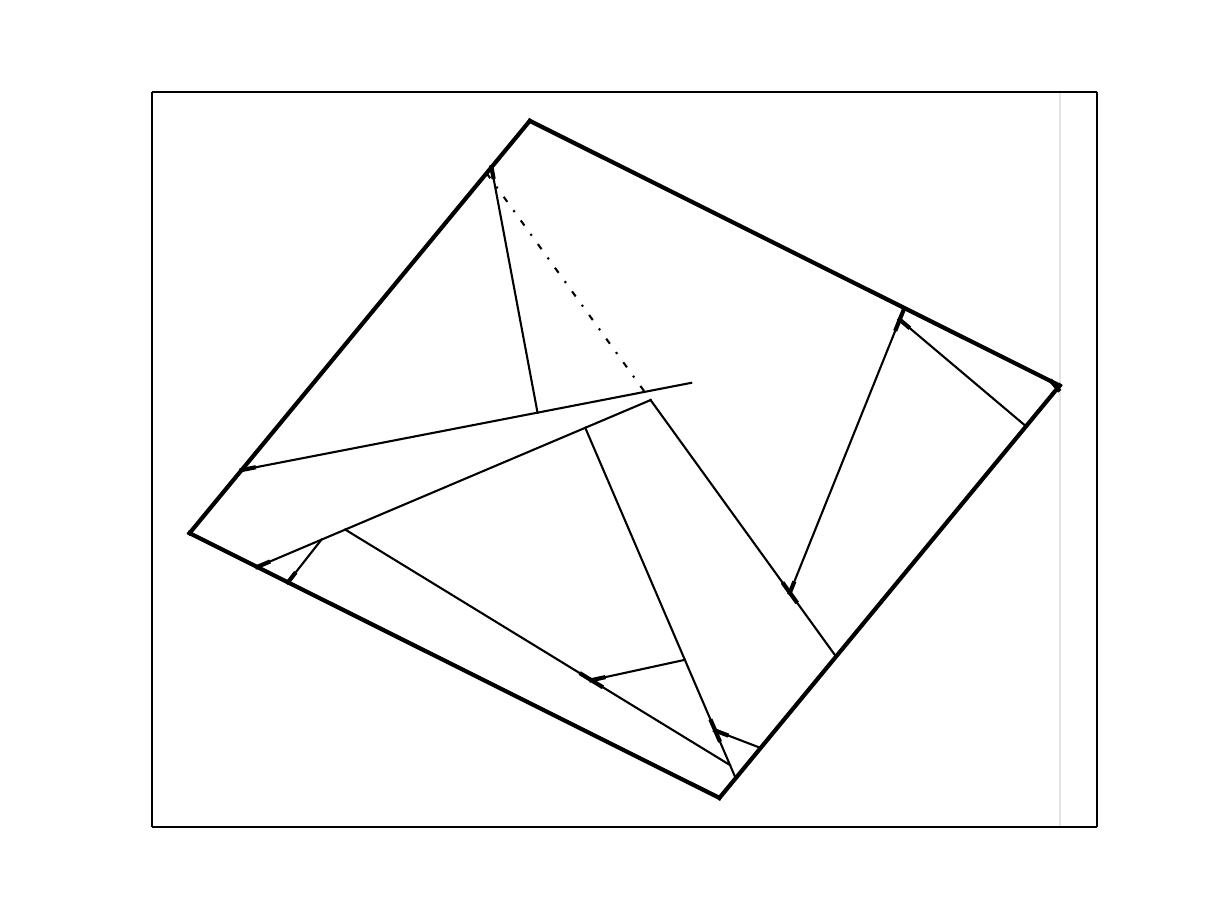}
             }
\subfloat[][]{
             \label{rec2n}
             \includegraphics[width=.49\textwidth]{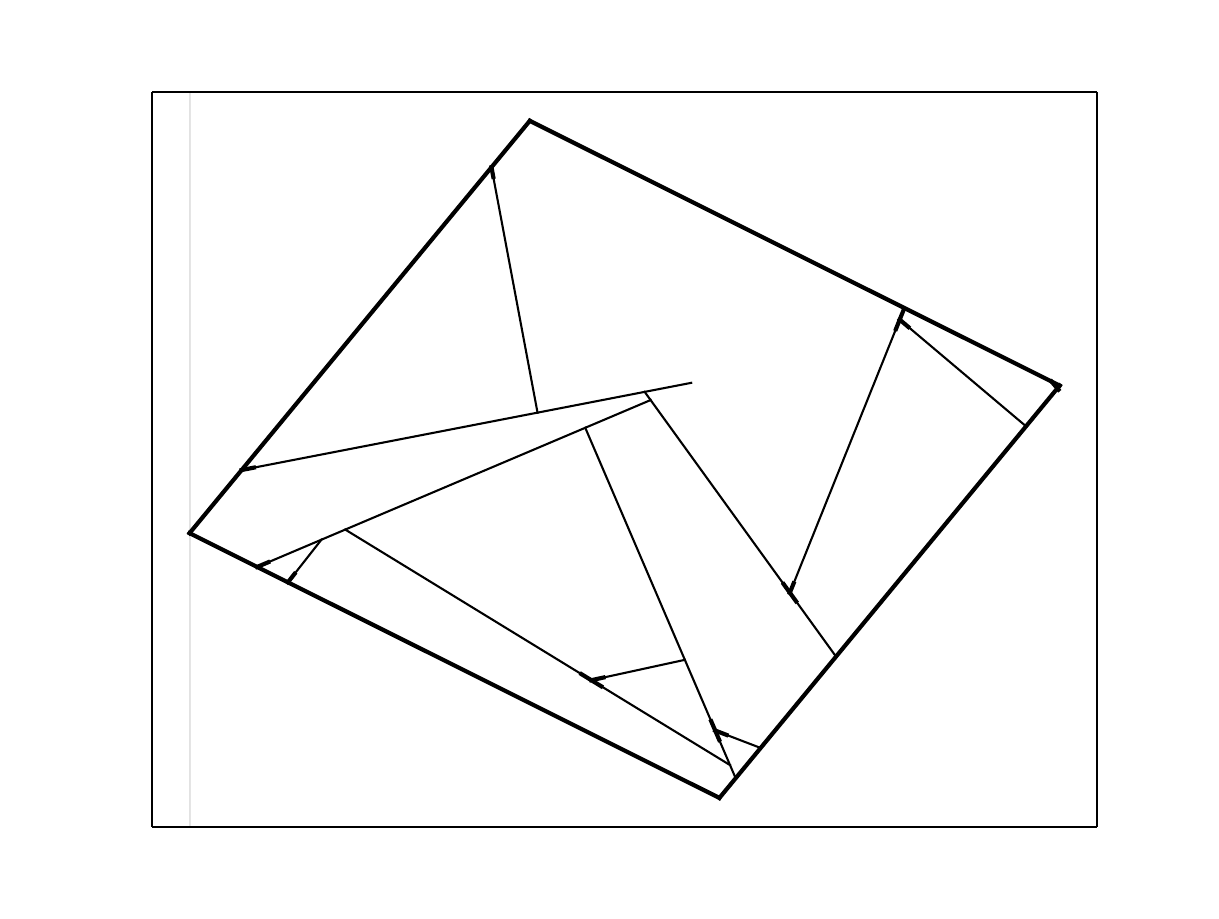}
             }\\
\subfloat[][]{
             \label{rec2o}
             \includegraphics[width=.49\textwidth]{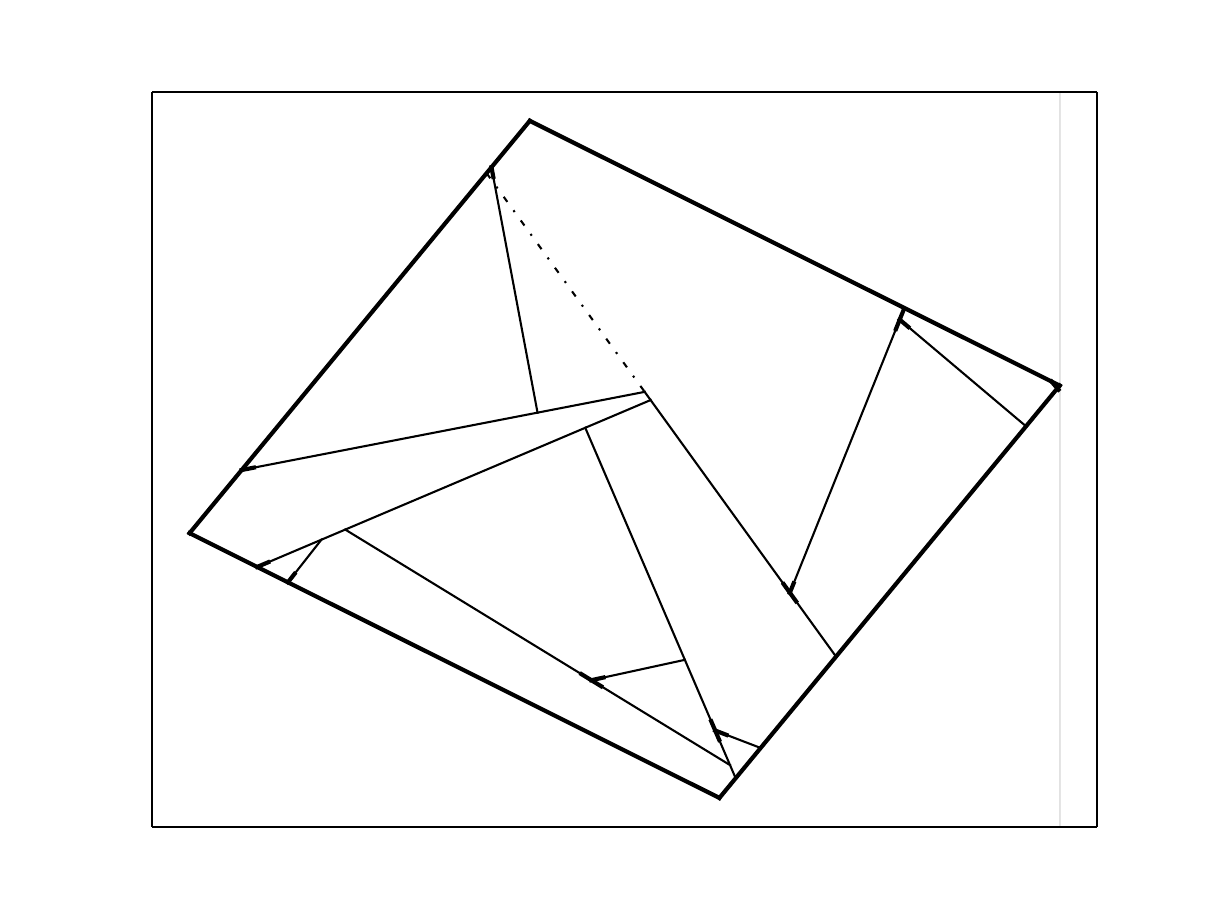}
             }
\subfloat[][]{
             \label{rec2p}
             \includegraphics[width=.49\textwidth]{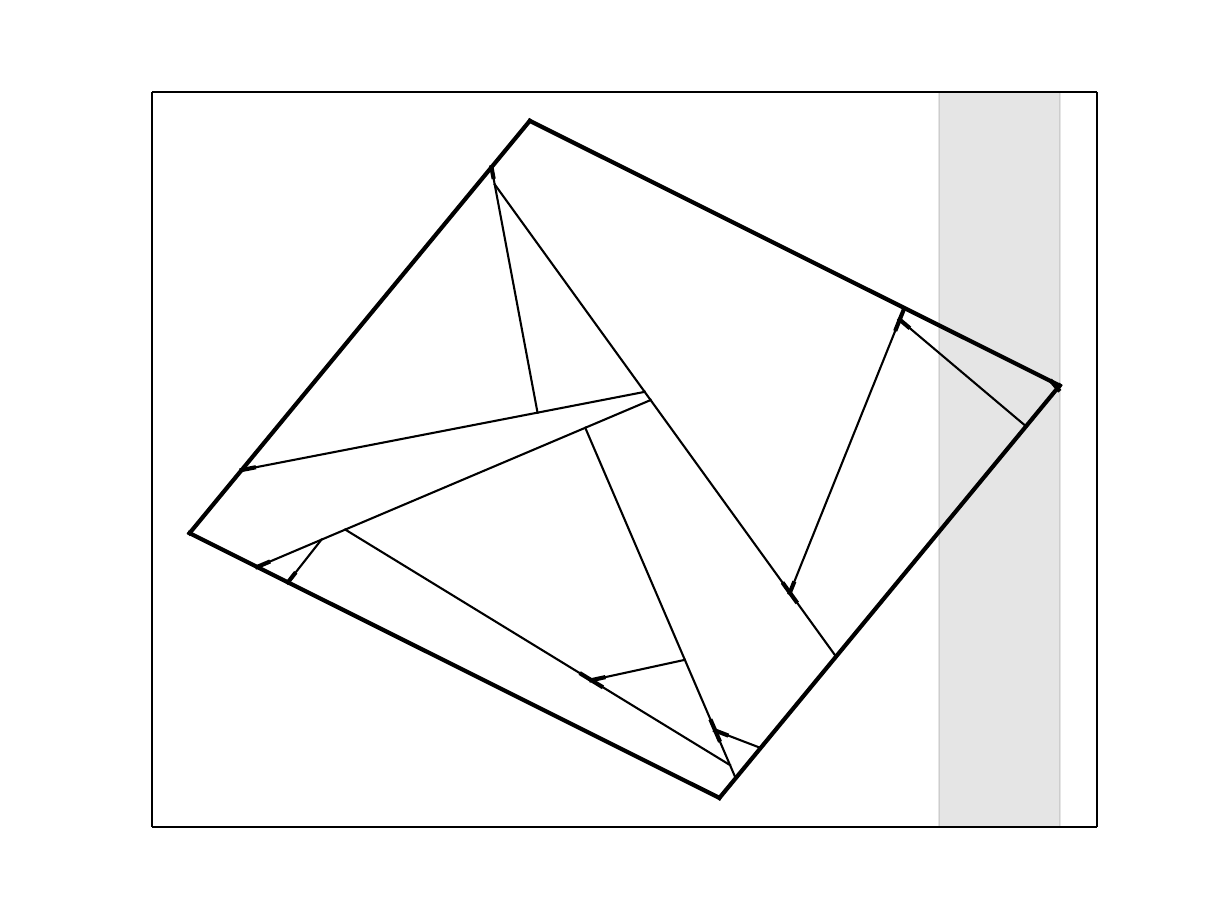}
             }
             \caption{Another pass backwards to extend the orphan \protect\subref{rec2k}. On the following forward pass, a line has one more child \protect\subref{rec2l}, so  cut \protect\subref{rec2m} is sooner than before \protect\subref{rec2i}. Next backwards: the orphan is extended \protect\subref{rec2n}. Next forward: the putative parent is cut \protect\subref{rec2o}. Next backwards and forward: the orphan is extended, and every line has the right number of orphan children; the algorithm stops \protect\subref{rec2p}.}
\end{figure}

\end{document}